\documentclass[11pt, a4paper]{amsart}

\usepackage{graphicx}
\usepackage[utf8]{inputenc}
\usepackage{amsmath}
\usepackage{amssymb}
\usepackage{hyperref}
\usepackage{amsthm}
\usepackage{enumerate}
\usepackage{wasysym}

\usepackage{extarrows}
\usepackage{polynom}
\usepackage{dsfont}
\usepackage{verbatim}
\usepackage[shortlabels]{enumitem}
\usepackage[toc,page]{appendix}
\usepackage{a4wide}
\usepackage{todonotes}

\usepackage{pst-node}

\usepackage[pagewise,mathlines]{lineno}

\theoremstyle{definition}
\newtheorem{theorem}{Theorem}[section]
\newtheorem{prop}[theorem]{Proposition}
\newtheorem{definition}[theorem]{Definition}

\newtheorem{remark}[theorem]{Remark}
\newtheorem{lemma}[theorem]{Lemma}
\newtheorem{coro}[theorem]{Corollary}
\numberwithin{equation}{section}

\newcommand{\abs}[1]{\left\lvert#1\right\rvert}
\newcommand{\norm}[1]{\left\|#1\right\|}

\newcommand{\R}{\mathbb R}

\newcommand{\Z}{\mathbb Z}
\newcommand{\C}{\mathbb C}
\renewcommand{\epsilon}{\varepsilon}
\renewcommand{\S}{\mathcal S}

\newcommand{\dx}{\mathrm d}

\newcommand{\tr}{\mathrm{Tr}}
\newcommand{\K}{\mathcal K}
\newcommand{\T}{\mathcal T}

\newcommand{\F}{\mathcal F}

\begin{document}

\allowdisplaybreaks
\title{Kinetic Maximal $L^2$-Regularity for the (fractional) Kolmogorov equation}

\date{\today}
\author{Lukas Niebel}
\email{lukas.niebel@uni-ulm.de}
\author{Rico Zacher$^*$}
\thanks{$^*$Corresponding author. The first author is supported by a graduate
scholarship (''Landesgraduiertenstipendium'') granted by the State of Baden-Wuerttemberg, Germany 
(grant number 1902 LGFG-E).}
\email[Corresponding author]{rico.zacher@uni-ulm.de}
\address[Lukas Niebel, Rico Zacher]{Institut f\"ur Angewandte Analysis, Universit\"at Ulm, Helmholtzstra\ss{}e 18, 89081 Ulm, Germany.}

\maketitle
{\centering\footnotesize Dedicated to Matthias Hieber on the occasion of his 60th birthday \par}

\begin{abstract}
We introduce the notion of kinetic maximal $L^2$-regularity with temporal weights for the (fractional) Kolmogorov equation. In particular, we determine the function spaces for the inhomogeneity and the initial value which characterize the regularity of solutions to the fractional Kolmogorov equation in terms of fractional anisotropic Sobolev spaces. It is shown that solutions of the homogeneous (fractional) Kolmogorov equation define a semi-flow in a suitable function space and the property of instantaneous regularization is investigated. 
\end{abstract}
\vspace{1em}

{\centering \textbf{AMS subject classification.} 35K65, 35B65, 35H10 \par}
\vspace{1em}

\section{Introduction} 
 
 The (fractional) Kolmogorov equation 
 	\begin{equation} \label{eq:intkol}
	\begin{cases}
		\partial_t u +v \cdot \nabla_x u + (-\Delta_v)^{\frac{\beta}{2}} u = f, \quad t>0 \\
		u(0) = g,
	\end{cases}
\end{equation}
 where $u = u(t,x,v)$ and $\beta \in (0,2]$ has gained more and more interest in the past years. This is mainly due to the following three phenomena. First, it can be seen as the prototype of a kinetic partial differential equation similar to the famous Boltzmann or the Landau equation. Second, even though the (fractional) Laplacian only acts in half of the variables, i.e.\ the equation is {\em degenerate}, solutions of this equation admit good regularity properties. Last, but not least, it serves as an excellent example to study the {\em regularity transfer} (from the $v$ to the $x$ variable), a special feature of kinetic equations.
 
If $\beta = 2$, the Kolmogorov equation models the density of moving particles under the assumption that the velocity is given by a Wiener process. The variable $x$ describes the position and the variable $v$ describes the velocity. The Kolmogorov equation has already been studied in \cite{kolmogoroff_zufallige_1934}, where Kolmogorov gave a fundamental solution in the case $\beta = 2$. Based on this work, H\"ormander investigated regularity properties of a much more general type of equations, namely those that satisfy the H\"ormander rank condition. He proved that every partial differential equation of this type is {\em hypoelliptic}, i.e.\ every distributional solution $u$ of such a PDE with smooth inhomogeneity $f$ must be smooth, too. First, local $L^p$ estimates for related equations have been studied in \cite{rothschild_hypoelliptic_1976}. Later, the tools for global estimates were provided in \cite{folland_estimates_1974}. In particular, it can be shown that for all $f \in L^2(\R^{2n+1})$ and any solution $u$ of the corresponding Kolmogorov equation for $\beta = 2$ the estimate
 \begin{equation} \label{eq:maxkinreg}
 	\norm{\partial_t u+ v \cdot \nabla_x u}_2 +\norm{\Delta_v u}_2  \lesssim \norm{f}_2
 \end{equation} 
 is satisfied. 
 
It is natural to ask whether one also gains some regularity in the position variable $x$. This question has been studied in the theory of kinetic equations with the help of velocity averages around the turn of the century. However, already in \cite{rothschild_hypoelliptic_1976} it was shown that in the case $\beta=2$ one gains $2/3$ of a derivative but only in the sense of a local $L^2$-estimate. In \cite{bouchut_hypoelliptic_2002} it is proven that one gains $2/3$ of a derivative in the position variable in terms of a global estimate, too. Moreover, this result holds true for every kinetic equation, for which it is known that the solution admits two derivatives in velocity $v$. This result nicely illustrates the regularity transfer phenomena observed for kinetic equations. In particular, these results apply to the fractional Kolmogorov equation, too. Here, one gains $\frac{\beta}{\beta+1}$ of a derivative in $x$. A global $L^2$-estimate for the fractional Kolmogorov equation similar to that in \eqref{eq:maxkinreg} was proven first in \cite{alexandre_fractional_2012}. A different proof, using a more stochastic language, can be found in \cite{chen_lp-maximal_2018}. Another proof in the case $\beta = 2$ can be also found in \cite{bouchut_hypoelliptic_2002}. Other important references, concerning the Kolmogorov equation, we want to mention are \cite{bramanti_invitation_2014,bramanti_global_2010,bramanti_global_2013,farkas_classofhypo_2008,lerner_hypoelliptic_2010,lanconelli_class_1994,
lorenzi_analytical_2017,lunardi_schauder_2005,bezard_regularite_1994}. Weak solutions of the Kolmogorov equation are considered in \cite{carrillo_global_1998,armstrong_variational_2019,golse_harnack_2016}. 
 
 The (abstract) theory of maximal $L^p$-regularity is closely connected to the $L^p$-theory of partial differential equations, see e.g.\
\cite{DHP}. To be more precise, let us consider the Cauchy problem 
 \begin{equation*}
 	 \partial_t u = Au+f,\quad u(0) = u_0,
 \end{equation*}
 where $A$ is a suitable linear operator, $f$ is the inhomogeneity in the space $L^p((0,T);X)$, $X$ is some Banach space and $u_0 \in X$ is the initial value. One is interested in a characterization of the data $f$ and $u_0$ in terms of function spaces such that $\partial_t u, Au \in L^p((0,T);X)$ holds. The choice of the space $X$ depends on the specific type of solution one is looking for. Let us sketch this for the heat equation
 \begin{equation*}
 	\begin{cases}
 		\partial_t u(t,x) = \Delta u(t,x), \quad t\in (0,T),\,x \in \R^n \\
 		u(0,x) = u_0(x), \quad x \in \R^n.
 		\end{cases}
\end{equation*}
  Choosing $X = L^2(\R^{n})$ we have that $\partial_t u$ and  $\Delta u \in L^2((0,T); L^2(\R^{n}))$ if and only if $f \in L^2((0,T); L^2(\R^{n}))$ and $u_0 \in H^1(\R^n)$, i.e.\ we are able to characterize assumptions such that the function $u$ is a strong solution of the equation. Moreover, one can show that in this case $u \in C([0,T];H^1(\R^n))$. The choice $X = H^{-1}(\R^n)$ leads to a characterization of weak solutions. We have $\partial_t u \in L^2((0,T);H^{-1}(\R^n))$ and $u \in L^2((0,T);H^1(\R^n))$ if and only if $f \in L^2((0,T);H^{-1}(\R^n))$ and $u_0 \in L^2(\R^n)$. If this is the case, here, one can also prove that $u \in C([0,T];L^2(\R^n))$. We point out that the theory of maximal $L^p$-regularity has been proven to be a very powerful tool to study nonlinear, more precisely, quasilinear variants of the equation $\partial_tu = Au$ such as $\partial_t u = A(u) u$, see e.g.\ the monograph
  \cite{pruss_moving_2016}. For an introduction to maximal $L^p$-regularity we refer to  \cite{amann_linear_1995,KunstWeis2004,pruss_moving_2016}.     
 
  It seems that a precise $L^p$-theory for the initial-value problem of the Kolmogorov equation in $(0,T) \times \R^{2n}$ has not yet been established, even in the special case $p=2$. Unfortunately, the Kolmogorov equation as well as other related kinetic equations do {\em not} enjoy the property of maximal $L^p$-regularity in the classical sense. This is indicated by estimate \eqref{eq:maxkinreg}, which does not provide a control for the time derivative itself. To remedy the lack of maximal $L^p$-regularity for the (fractional) Kolmogorov equation and restricting ourselves to the more accessible case $p=2$, we introduce the concept of {\em kinetic maximal $L^2$-regularity} in this article. To do so, we need to first determine a suitable space to measure the regularity of the solution $u$. The estimate \eqref{eq:maxkinreg} (in the case $\beta=2$) suggests that one should consider functions $u \in L^2((0,T);X)$ satisfying $\partial_t u + v \cdot \nabla_x u \in L^2((0,T);X)$ and $\Delta_v u \in L^2((0,T);X)$, where $X$ is an appropriate base space w.r.t.\ the spatial variables. 
It turns out that this is indeed a good choice. For example, as a very special case of our main result, Theorem \ref{thm:maxreg}, we are able to obtain the following result.
   
   \begin{theorem} \label{thm:thmint}
   		Let $T>0$ and $\beta \in (0,2]$. The Kolmogorov equation possesses a unique solution $u \in L^2((0,T);L^2(\R^{2n}))$ satisfying 
   		$$\partial_t u + v \cdot \nabla_x u \in L^2((0,T);L^2(\R^{2n})), \; (-\Delta_v)^{\frac{\beta}{2}} u \in L^2((0,T);L^2(\R^{2n})) $$
   		and $u \in C([0,T];H_x^{\frac{\beta/2}{\beta+1}}(\R^{2n}) \cap H_v^{\beta/2}(\R^{2n})$ if and only if 
 \[  		
   		f \in L^2((0,T);L^2(\R^{2n}))\quad \mbox{and}\quad u_0 \in H_x^{\frac{\beta/2}{\beta+1}}(\R^{2n}) \cap H_v^{\beta/2}(\R^{2n}).
  \]
   \end{theorem} 
   
   Here, $H_x^{\frac{\beta/2}{\beta+1}}(\R^{2n}) $ and $H_v^{\beta/2}(\R^{2n})$ denote fractional Sobolev spaces in the respective variables $x$ and $v$. A precise definition can be found in the next section. To prove the $L^2$-estimates, we use Plancherel's theorem and make excessive use of the Fourier transformation. The continuity of solutions with values in $L^2$ can be deduced by considering the characteristics, i.e. $(t,x,v) \mapsto (t,x+tv,v)$, corresponding to the kinetic first order term $\partial_t + v \cdot \nabla_x$. 
   
Theorem \ref{thm:thmint} characterizes strong $L^2$-solutions. In our main result, Theorem \ref{thm:maxreg}, we significantly extend this theorem in two different directions. First, instead of $X=L^2(\R^{2n})$, we consider a whole {\em scale of anisotropic fractional Sobolev spaces}, which also allows to study both weak solutions and strong solutions with higher regularity. Second, we introduce {\em temporal weights} of the form $t^{2(1-\mu)}$ which enable us to lower the initial value regularity and to prove results on the instantaneous gain of regularity for the homogeneous Kolmogorov equation. In particular, we are able to show that in our general framework, any solution becomes $C^\infty$-smooth in time and space for $t>0$, see Theorem \ref{TheoremSmoothing}.
   
    Due to the linearity of the problem we are able to study the nonhomogeneous problem with vanishing initial data and the initial-value problem with $f=0$ separately. We remark that if $u_0  = 0$ the statement of Theorem \ref{thm:thmint} is a consequence of the well-known results on global $L^2$-estimates, see e.g.\ \cite{chen_lp-maximal_2018}, which also covers the case $p\neq 2$. Restricting to $p=2$, our argument generalizes to the one given in
\cite{chen_lp-maximal_2018} inasmuch as we allow for temporal weights of the form $t^{2(1-\mu)}$ with $\mu \in (1/2,1]$ and for different base spaces $X$ other than $L^2$. The result for the initial value seems to be new, even in the 
special case of Theorem \ref{thm:thmint} with $\beta=2$. We point out that we do not only identify the trace
space for the initial value but also prove continuity of the solutions in the trace space. This is a crucial property necessary for the homogeneous equation to induce a {\em semi-flow}.   
   
We note that the results of this paper, at least in case of strong solutions, can be extended to $L^p$-solutions of the Kolmogorov equation. This is work in progress and will be subject of the forthcoming article \cite{niebel_kinetic_nodate-1}. The $L^p$-setting is much more involved as Plancherel's theorem cannot be used
anymore. Instead one needs to use more sophisticated tools such as Littlewood-Paley decompositions and work also 
with anisotropic Besov spaces.  
	
	The plan of the present article is as follows. We first fix some notation and prove some preliminary results in Section \ref{sec:prelim}. Furthermore, we introduce the representation formula for solutions of the Kolmogorov equation in this section. Then, we will prove $L^2$-estimates for the nonhomogeneous Kolmogorov equation with vanishing initial data and the homogeneous Kolmogorov equation in Section \ref{sec:inhomo} and \ref{sec:homo}, respectively. In Section \ref{sec:maxkinL2}, we then introduce the concept of kinetic maximal $L^2$-regularity with temporal weights and prove that the fractional Kolmogorov equation satisfies this property. Here, we also provide a deeper study of the involved function spaces and justify why solutions of the homogeneous Kolmogorov equation define a semi-flow. Finally, in Section \ref{sec:regCinf}, we will investigate the gain of regularity for the homogeneous Kolmogorov equation and prove the instantaneous $C^\infty$-regularization of any solution to the Kolmogorov equation.

\section{Preliminaries}
\label{sec:prelim}
We are interested in the regularity of measurable functions $u \colon [0,\infty) \times \R^{2n} \to \R$, $u = u(t,x,v)$ which are solutions (at least in the distributional sense) of the (fractional) Kolmogorov equation
\begin{equation} \label{eq:kol}
	\begin{cases}
		\partial_t u +v \cdot \nabla_x u = -(-\Delta_v)^{\frac{\beta}{2}} u +f, \quad t>0 \\
		u(0) = g,
	\end{cases}
\end{equation}
where $\beta \in (0,2]$ and the data $f$ and $g$ are given. The Fourier transform in $(x,v)$ with respective Fourier variables $(k,\xi)$ of a function $u$ will be denoted by $\hat{u}$. At least formally applying the Fourier transform to equation \eqref{eq:kol} gives
\begin{equation*}
	\begin{cases}
		\partial_t \hat{u} - k  \cdot \nabla_\xi \hat{u} = -\abs{\xi}^\beta \hat{u} + \hat{f}, \quad t>0 \\
		\hat{u}(0) = \hat{g}.
	\end{cases}
\end{equation*}
In the following we will use the function $e_\beta \colon [0,\infty) \times \R^{2n} \to \R$ given by
\begin{equation*}
	e_\beta(t,k,\xi) = \exp( -\int_0^t \abs{\xi+(t-r) k}^\beta \dx r) = \exp( -\int_0^t \abs{\xi+\sigma k}^\beta \dx \sigma)
\end{equation*}
for $\beta \in (0,2]$. A direct calculation shows that if $\beta = 2$ the integral simplifies to 
\begin{equation*}
	e_2(t,k,\xi) = \exp\left(-\abs{\xi}^2t- \xi \cdot k t^2 - \abs{k}^2\frac{t^3}{3}\right).
\end{equation*}
A solution to the Fourier transformed Kolmogorov equation can be explicitly given by means of the method of characteristics as
\begin{equation} \label{eq:fouriersol}
	\hat{u}(t,k,\xi) = \hat{g}(k,\xi + tk) e_\beta(t,k,\xi) + \int_0^t \hat{f}(s,k,\xi+(t-s)k) e_\beta(t-s,k,\xi) \dx s,
\end{equation}
for sufficiently nice functions $f$ and $g$. It is well-known that there exists a strongly continuous semigroup in $L^2(\R^{2n})$ associated to the fractional Kolmogorov equation, which we will denote by $T(t)$. A thorough treatment of the fractional Kolmogorov semigroup in $L^2(\R^{2n})$ can be found in \cite{alphonse_smoothing_2018}.

\noindent We are going to study the connection between moment bounds of the Fourier transformed solutions $\hat{u}$ and moment bounds on the Fourier transformed initial value $\hat{g}$ as well as the inhomogeneity $\hat{f}$. Moment bounds in $\xi$ and $k$ of the Fourier transformed functions  $\hat{u}$, $\hat{f}$ and $ \hat{g}$ correspond to differentiability properties of the functions $u$, ${f}$ and $g$. Let us introduce the following (kinetic) fractional Sobolev spaces. For $s \in \R$ we define one space to measure the regularity in the position variable $x$
\begin{equation*}
	H_x^s(\R^{2n}) = \left\{ f \in \S'(\R^{2n}) \; \colon \; (1+\abs{k}^{2})^\frac{s}{2} \F(f)(k,\xi) \in L^2(\R^{2n}),   \right\}
\end{equation*}
and one for the regularity in the velocity variable $v$ 
\begin{equation*}
	H_v^s(\R^{2n}) = \left\{ f \in \S'(\R^{2n}) \; \colon \; (1+\abs{\xi}^{2})^\frac{s}{2} \F(f)(k,\xi) \in L^2(\R^{2n})   \right\}.
\end{equation*}
Both of these spaces are Hilbert spaces when equipped with the obvious scalar product. In the following we denote by $D_x^s = (-\Delta_x)^\frac{s}{2}$ and $D_v^s = (-\Delta_v)^\frac{s}{2}$ the fractional Laplacian in the variables $x$ and $v$, respectively. By $\dot{H}_x^s(\R^{2n})$ and $\dot{H}_v^s(\R^{2n})$ we denote the corresponding homogeneous Sobolev spaces, which are defined as the closure of Schwartz functions modulo polynomials with respect to the $\norm{D_x^s \cdot}_{\dot{L}^2}$ and $\norm{D_v^s \cdot}_{\dot{L}^2}$ norm respectively. Here, $\dot{L}^2$ denotes the $L^2$-space modulo polynomials, see \cite[Section 2.4]{sawano_theory_2018}. Moreover, we introduce the following scale of anisotropic Sobolev spaces
\begin{equation*}
	X^s_\beta =  \left\{ f \in \S'(\R^{2n}) \; \colon \; \left(  (1+\abs{\xi}^{2})^\frac{\beta}{2}+(1+\abs{k}^2)^\frac{\beta}{2(\beta+1)} \right)^{s} \F(f)(k,\xi) \in L^2(\R^{2n})   \right\}
\end{equation*}
equipped with the respective norm $\norm{\cdot}_{X^s_\beta}$. Again, we denote by $\dot{X}^s_\beta$ the corresponding homogeneous space. Instead of the smooth multiplier used in the definition of the spaces $X^s_\beta$
we can equivalently use powers of the multiplier
$\omega_\beta \colon \R^{2n} \to [0,\infty)$, defined as
\begin{equation*}
	\omega_\beta(k,\xi) = 1+\abs{k}^\frac{\beta}{\beta+1}+\abs{\xi}^\beta
\end{equation*}
for $k,\xi \in \R^n$, to characterize the norm of ${X}^s_\beta$, by Plancherel's theorem. This multiplier will play an important role throughout the article. We note that if $s \ge 0$, then $H_x^{s\frac{\beta}{\beta+1}}(\R^{2n}) \cap H_{v}^{s\beta}(\R^{2n}) \cong X^s_\beta $.   

 Given any Banach space $X$, $T \in (0,\infty]$ and $\mu \in (-\infty,1]$ we define the time-weighted $L_\mu^2$ space as
\begin{equation*}
	L_\mu^2((0,T);X) = \{ f \colon (0,T) \to X \colon f \text{ measurable and } \int_0^T t^{2-2\mu} \norm{f(t)}_X^2 \dx t < \infty \}.
\end{equation*}
Equipped with the norm $\norm{f}_{2,\mu,X}^2 = \int_0^T t^{2-2\mu} \norm{f(t)}_X^2 \dx t$, the vector space $L_\mu^2((0,T);X)$ is a Banach space.

In the following calculations we denote by the letter $c$ a generic positive constant which may change from line to line. For two functions $f,g$, the notation $f\lesssim g$ then means that $f\le cg$ on the respective domain. Note that
in our arguments, estimates are always proven first for smooth functions, the general case then follows by an approximation argument.

The following estimate on $e_\beta$ will be a useful tool in later proofs.  
\begin{lemma} \label{lem:upperbounde}
	For all $\beta \in (0,2]$ there exist constants $c_1 =c_1(\beta), {c}_2 = {c}_2(\beta)>0$ such that we have
	\begin{equation*}
		\exp(-c_1\abs{\xi+tk}^{\beta}t- c_1\abs{k}^{\beta}t^{\beta+1}) \le e_\beta(t,k,\xi) \le \exp(-c_2\abs{\xi+tk}^\beta t- c_2\abs{k}^\beta t^{\beta+1})
	\end{equation*}
	for all $t \ge 0$ and all $\xi,k \in \R^n$.
\end{lemma}

\begin{proof}
	The case $\beta  = 2$ is much easier to prove. It can be directly shown by using the Cauchy-Schwarz and Young's inequality.	Let $\beta \in (0,2]$. Observe that, substituting $\xi = \tilde{\xi}+tk$, it suffices to estimate 
	\begin{equation*}
		\int_0^t \abs{\xi-rk}^\beta \dx r \ge c\abs{\xi}^\beta t + c \abs{k}^\beta t^{\beta+1}
	\end{equation*}
	for some constant $c>0$ to show the estimate from above. If $k = 0$ the estimate follows directly by calculating the integral on the left hand side.  If $t \le \frac{\abs{\xi}}{\abs{k}}$, then by the inverse triangle inequality we have
	\begin{equation*}
		\int_0^t \abs{\xi-rk}^\beta \dx r \ge \int_0^t (\abs{\xi}-r\abs{k})^\beta \dx r = \frac{\abs{\xi}^{\beta+1}}{(\beta+1) \abs{k}} - \frac{(\abs{\xi}-t\abs{k})^{\beta+1}}{(\beta+1) \abs{k}}.
	\end{equation*}
	If in addition $\frac{\abs{\xi}}{2\abs{k}} \le t \le  \frac{\abs{\xi}}{\abs{k}}$, then the inequality
	\begin{equation*}
		\int_0^t \abs{\xi-rk}^\beta \dx r \ge (1-\frac{1}{2^{\beta+1}})\frac{\abs{\xi}^{\beta+1}}{(\beta+1) \abs{k}} \ge c\abs{\xi}^\beta t
	\end{equation*}
	follows from the latter inequality. If $t \le \frac{\abs{\xi}}{2\abs{k}} $, then 
	\begin{equation*}
		\int_0^t \abs{\xi-rk}^\beta \dx r \ge \frac{1}{2^\beta} \int_0^t \abs{\xi}^\beta \dx r = c{\abs{\xi}^\beta}t. 
	\end{equation*}	
	Moreover, if $t \le \frac{\abs{\xi}}{\abs{k}}$, then $\abs{\xi}^\beta t \ge \abs{k}^\beta t^{\beta+1}$ and hence we have the desired lower estimate by $\abs{k}^\beta t^{\beta+1}$ in this case, too. 
	
	In the case that $t \ge \frac{\abs{\xi}}{\abs{k}}$ we estimate
	\begin{align*}
		\int_0^t \abs{\xi-rk}^\beta \dx r &= \int_{\frac{\abs{\xi}}{\abs{k}}}^t \abs{\xi-rk}^\beta  \dx r +\int_0^{\frac{\abs{\xi}}{\abs{k}}} \abs{\xi-rk}^\beta  \dx r \\
		&\ge \int_{\frac{\abs{\xi}}{\abs{k}}}^t (r \abs{k}-\abs{\xi})^\beta  \dx r +\int_0^{\frac{\abs{\xi}}{\abs{k}}} (\abs{\xi}-r \abs{k})^\beta   \dx r \\
		&= \frac{1}{(\beta+1)\abs{k}} \left[ (t\abs{k}-\abs{\xi})^{\beta+1}-0-0+\abs{\xi}^{\beta+1} \right] \\
		&\ge \frac{2^{-\beta}}{\beta+1} \abs{k}^\beta t^{\beta+1} = c\abs{k}^\beta t^{\beta+1},
	\end{align*}
	using the inequality $a^{\beta+1} \le (a-b+b)^{\beta+1} \le 2^\beta ((a-b)^{\beta+1}+b^{\beta+1})$.
	Again, as $\abs{k}^\beta t^{\beta+1} \ge \abs{\xi}^\beta t$ and arguing as above we conclude the desired estimate from above. The estimate from below follows  using the triangle inequality and a straightforward integration.
\end{proof}

\begin{lemma} \label{lem:decayestimsemi}
	Let $\beta \in (0,2]$, $s\ge -1$ and $T>0$. Then for all $k,\xi \in \R^n$ and any $t \in (0,T]$ we have	
	\begin{equation*}
		\omega_\beta(k,{\xi-tk})^{2(s+1)} e_\beta(t,k,\xi-tk) \le c(T) \frac{\omega_\beta(k,\xi)^{2s}}{t^2}.
	\end{equation*}
\end{lemma}

\begin{proof}
	As $s \ge -1$ we have
	\begin{equation*}
		\omega_\beta(k,\xi)^{2(s+1)} \lesssim 1+\abs{\xi}^{2(s+1)\beta}+ \abs{k}^{2(s+1)\frac{\beta}{\beta+1}}.
	\end{equation*}
	If $t\abs{k} \le \abs{\xi}$, then 
	\begin{align*}
		\abs{\xi-tk}^{2s\beta+2\beta}e_\beta(t,k,\xi-tk) &\lesssim \abs{\xi}^{2s\beta}t^{-2} (t\abs{\xi}^\beta)^{2} \exp(-c\abs{\xi}^\beta t) \\
		&\le \abs{\xi}^{2s\beta}t^{-2} \sup_{x>0} x^2\exp(-cx) \lesssim \abs{\xi}^{2s\beta}t^{-2},
	\end{align*}
	as a consequence of Lemma \ref{lem:upperbounde}. In the case $t\abs{k} \ge \abs{\xi}$ we deduce
	\begin{align*}
		\abs{\xi-tk}^{2s\beta+2\beta}e_\beta(t,k,\xi-tk) &\lesssim \abs{k}^{2s\frac{\beta}{\beta+1}} t^{-2}   \left( t^{2((s+1)\beta+1)}\abs{k}^{2((s+1)\beta+1)\frac{\beta}{\beta+1}} \exp(-c\abs{k}^\beta t^{\beta+1}) \right) \\
		&= \abs{k}^{2s\frac{\beta}{\beta+1}} t^{-2}   \left( \abs{k}^{\beta}t^{\beta+1}\right)^\frac{2((s+1)\beta+1)}{\beta+1} \exp(-c\abs{k}^\beta t^{\beta+1}) \\
		&\le \abs{k}^{2s\frac{\beta}{\beta+1}} t^{-2} \sup_{x>0} x^{\frac{2((s+1)\beta+1)}{\beta+1}} \exp(-cx) \\
		&\lesssim \abs{k}^{2s\frac{\beta}{\beta+1}} t^{-2}.
	\end{align*}
	Moreover, we have $1 \le c(T)t^{-2}$ for some constant $c = c(T)$. For $s \ge 0$ we have $\omega_\beta(k,\xi)^{2s} \approx 1+\abs{\xi}^{2s\beta}+ \abs{k}^{2s\frac{\beta}{\beta+1}}$, which shows the claim. 
	
	For $s<0$ we need to argue differently. First, we observe that
	\begin{align*}
		\frac{\omega_\beta(k,\xi-tk)^{2(s+1)}}{\omega_\beta(k,\xi)^{2s}} &\lesssim \abs{\xi-tk}^{2(s+1)\beta}\abs{\xi}^{-2\beta s}+ \abs{k}^{2\frac{\beta}{\beta+1}} \\
		&+ \abs{\xi-tk}^{2(s+1)\beta} \abs{k}^{-2s\frac{\beta}{\beta+1}}+\abs{k}^{2(s+1)\frac{\beta}{\beta+1}}\abs{\xi}^{-2\beta s} \\
		&+1+\abs{\xi}^{-2\beta s}+\abs{k}^{-2s\frac{\beta}{\beta+1}}+\abs{\xi-tk}^{2(s+1)\beta}+\abs{k}^{2(s+1)\frac{\beta}{\beta+1}}.
	\end{align*}
	Again, distinguishing the cases $t\abs{k} \le \abs{\xi}$ and $t\abs{k} \ge \abs{\xi}$ we can treat each of these terms. We show this for the first term in the case $t\abs{k} \ge \abs{\xi}$. We have
	\begin{align*}
		\abs{\xi-tk}^{2(s+1)\beta}\abs{\xi}^{-2\beta s}e_\beta(t,k,\xi-tk) &\lesssim t^{2(s+1)\beta}\abs{k}^{2(s+1)\beta} \exp(-c\abs{k}^\beta t^{\beta+1}) \abs{\xi}^{-2\beta s} \exp(-c\abs{\xi}^\beta t) \\
		&\lesssim t^{-2s-2}t^{2s} = t^{-2}.
	\end{align*}
	The other estimates follow similarly, where the last five terms can be estimated by $c(T)t^{-2}$ only.
\end{proof}

\section{Estimates in the nonhomogeneous case with vanishing initial data}
\label{sec:inhomo}

In this section we prove $L^2$-estimates for functions $u$ given by equation \eqref{eq:fouriersol} in the case of zero initial value $g = 0$.

\begin{prop} \label{prop:inhomogeneity}
	Let $\beta \in (0,2]$, $s \ge -1/2$, $\mu \in (\frac{1}{2},1]$ and $T \in (0,\infty)$.  For every $f \in L^2_\mu((0,T);{X}^s_\beta)$, the function $u$ given by equation \eqref{eq:fouriersol} with $g = 0$ satisfies $u \in L^2_\mu((0,T);{X}^{s+1}_\beta)$.
	In particular, the estimate 
	\begin{align*}
		\int_{0}^T \int_{\R^{2n}} t^{2-2\mu} \omega_\beta(k,\xi)^{2(s+1)}\abs{\hat{u}(t,k,\xi)}^2 \dx k \dx \xi \dx t \le c \int_{0}^T \int_{\R^{2n}} t^{2-2\mu} \omega_\beta(k,\xi)^{2s}\abs{\hat{f}(t,k,{\xi})}^2 \dx k \dx {\xi}   \dx t
	\end{align*}
	is satisfied for some constant $c = c(s,\beta,T) > 0$.
\end{prop}

\begin{proof}
	Let us first note that we can write 
	\begin{align*}
		\hat{u}(t,k,\xi) &= t^{\mu-1} \int_0^t e_\beta(t-\tau,k,\xi)\tau^{1-\mu}\hat{f}(\tau,k,\xi+(t-\tau)k) \dx \tau \\
		&\hphantom{=}+t^{\mu-1} \int_0^t e_\beta(t-\tau,k,\xi)\left( (t/\tau)^{1-\mu}-1 \right)\tau^{1-\mu}\hat{f}(\tau,k,\xi+(t-\tau)k) \dx \tau \\
		&=:\hat{u}_1(t,k,\xi)+\hat{u}_2(t,k,\xi)
	\end{align*}
	and that the second term vanishes in the unweighted case $\mu = 1$. It suffices to estimate each term separately. This decomposition goes back to \cite{pruss_maximal_2004} and proves to be useful when dealing with temporal weights. We use the Cauchy-Schwarz inequality to estimate
	\begin{align*}
		&\int_{0}^T \int_{\R^n} \int_{\R^n} t^{2-2\mu}  \omega_\beta(k,\xi)^{2(s+1)}\abs{\hat{u}_1(t,k,\xi)}^2 \dx k \dx \xi \dx t \\
		&\le \int_{0}^T \int_{\R^n} \int_{\R^n} \int_0^t \omega_\beta(k,\xi)^{2s+1}e_\beta(t-\tau,k,\xi)\tau^{2-2\mu}\abs{\hat{f}(\tau,k,\xi + (t-\tau)k)}^2 \dx \tau \, \eta(t,k,\xi) \dx k \dx \xi \dx t,
	\end{align*}
	where
	\begin{equation*}
		\eta(t,k,\xi) = \omega_\beta(k,\xi)\int_0^t  e_\beta(t-\tau,k,\xi)  \dx \tau.
	\end{equation*}
	We have
	\begin{align*}
		\abs{\xi}^\beta \int_0^t  e_\beta(\tau,k,\xi) \dx \tau &\le  \abs{\xi}^\beta\int_0^\infty e_\beta(\tau,k,\xi) \dx \tau \le \abs{\xi}^\beta\int_0^\infty \exp(-c\abs{\xi}^\beta \tau) \dx \tau = c
	\end{align*}
	for all $\xi,k \in \R^n$ by Lemma \ref{lem:upperbounde} and an explicit calculation of the last integral. Similarly, we calculate
	\begin{align*}
		\abs{k}^{\frac{\beta}{\beta+1}} \int_0^t  e_\beta(\tau,k,\xi) \dx \tau &\le  \abs{k}^{\frac{\beta}{\beta+1}} \int_0^\infty  e_\beta(\tau,k,\xi) \dx \tau \lesssim \abs{k}^{\frac{\beta}{\beta+1}} \int_0^\infty  \exp(-c\abs{k}^\beta \tau^{\beta+1}) \dx \tau \\
		&=  \int_0^\infty  \exp(-c\tau^{\beta+1}) \dx \tau =c(\beta)
	\end{align*}
	and
	\begin{equation*}
		\int_0^t e_\beta(\tau,k,\xi) \dx \tau \le T.
	\end{equation*}
	The latter estimates combined give
	\begin{equation*}
		\eta(t,k,\xi) = \omega_\beta(k,\xi)\int_0^t  e_\beta(t-\tau,k,\xi)  \dx \tau \le c(\beta,T).
	\end{equation*}
	Using Fubini's theorem we deduce \allowdisplaybreaks
	\begin{align*}
		&\int_{0}^T \int_{\R^n} \int_{\R^n} t^{2-2\mu} \omega_\beta(k,\xi)^{2(s+1)}\abs{\hat{u}_1(t,k,\xi)}^2 \dx k \dx \xi \dx t \\
		&\lesssim  \int_{\R^n} \int_{0}^T   \int_0^t \int_{\R^n} \omega_\beta(k,\xi)^{2s+1}e_\beta(t-\tau,k,\xi)\tau^{2-2\mu}\abs{\hat{f}(\tau,k,\xi + (t-\tau)k)}^2 \dx \xi \dx \tau   \dx t \dx k \\
		&=  \int_{\R^n} \int_{0}^T   \int_0^t \int_{\R^n} \omega_\beta(k,{\xi}-(t-\tau)k)^{2s+1} e_\beta(t-\tau,k,{\xi}-(t-\tau)k)\tau^{2-2\mu}\abs{\hat{f}(\tau,k,{\xi})}^2 \dx {\xi} \dx \tau   \dx t \dx k \\
		&= \int_{\R^n} \int_{0}^T   \int_\tau^T \int_{\R^n} \omega_\beta(k,{\xi}-(t-\tau)k)^{2s+1}  e_\beta(t-\tau,k,{\xi}-(t-\tau)k)\tau^{2-2\mu}\abs{\hat{f}(\tau,k,{\xi})}^2 \dx {\xi} \dx t   \dx \tau \dx k \\
		&\le  \int_{\R^n} \int_{0}^T   \int_0^T \int_{\R^n} \omega_\beta(k,{\xi}-rk)^{2s+1} e_\beta(r,k,{\xi}-rk)\tau^{2-2\mu}\abs{\hat{f}(\tau,k,{\xi})}^2 \dx {\xi} \dx r \dx \tau \dx k \\
		&= \int_{\R^n}\int_{\R^n} \int_{0}^T   \int_{\frac{\abs{\xi}}{\abs{k}} \land T}^T  \omega_\beta(k,{\xi}-rk)^{2s+1} e_\beta(r,k,{\xi}-rk) \dx r \; \tau^{2-2\mu}\abs{\hat{f}(\tau,k,{\xi})}^2  \dx k  \dx \xi \dx \tau \\
		&\hphantom{=}+ \int_{\R^n}\int_{\R^n} \int_{0}^T   \int_0^{\frac{\abs{\xi}}{\abs{k}} \land T} \omega_\beta(k,{\xi}-rk)^{2s+1} e_\beta(r,k,{\xi}-rk) \dx r \; \tau^{2-2\mu}\abs{\hat{f}(\tau,k,{\xi})}^2 \dx k  \dx \xi \dx \tau \\
		&=: \int_{\R^n}\int_{\R^n} \int_{0}^T [I_1(k,\xi) + I_2(k,\xi)]\tau^{2-2\mu}\omega_\beta(k,\xi)^{2s}\abs{\hat{f}(\tau,k,{\xi})}^2  \dx \tau \dx \xi \dx k,
	\end{align*}
	where
	\begin{equation*}
		I_1(k,\xi) =  \omega_\beta(k,\xi)^{-2s}\int_{\frac{\abs{\xi}}{\abs{k}} \land T}^T  \omega_\beta(k,{\xi}-rk)^{2s+1} e_\beta(r,k,{\xi}-rk) \dx r
	\end{equation*}
	and
	\begin{equation*}
		I_2(k,\xi) = \omega_\beta(k,\xi)^{-2s}\int_0^{\frac{\abs{\xi}}{\abs{k}} \land T} \omega_\beta(k,{\xi}-rk)^{2s+1} e_\beta(r,k,{\xi}-rk) \dx r.
	\end{equation*}
	We consider the first integral and make the following observation. If $s \ge -1/2$, then
	\begin{equation*}
		\abs{\xi-rk}^{\beta(2s+1)} \lesssim r^{\beta(2s+1)}\abs{k}^{\beta(2s+1)}
	\end{equation*}
	as $\abs{\xi} \le \abs{k}r$, whence, considering the cases $s \ge 0$ and $s<0$ separately it follows that
	\begin{align*}
		&\omega_\beta(k,{\xi}-rk)^{2s+1}\omega_\beta(k,\xi)^{-2s} \\ &\lesssim \begin{cases}
			1+\abs{k}^{\frac{\beta}{\beta+1}} +r^{\beta(2s+1)}\abs{k}^{\beta(2s+1)-2s\frac{\beta}{\beta+1}}, & s \ge 0 \\
			 1+\abs{k}^{\frac{\beta}{\beta+1}} +r^{\beta}\abs{k}^{\beta},& -1/2 \le s < 0.
		\end{cases} 
	\end{align*}
	In combination with Lemma \ref{lem:upperbounde} this observation shows, that it suffices to calculate the integral
	\begin{align*}
		\int_0^\infty r^a \exp(-c\abs{k}^\beta r^{\beta+1}) \dx r = \abs{k}^{-(a+1)\frac{\beta}{\beta+1}} \int_0^\infty s^a \exp(-cs^{\beta+1}) \dx s= c(\beta,s,a)\abs{k}^{-(a+1)\frac{\beta}{\beta+1}}
	\end{align*}
	for $a > -1 $. Indeed, choosing $a = 0$ and $a = (2s+1)\beta$ we deduce that $I_1(k,\xi) \le c(s,\beta,T)$ if $s \ge 0$. If $-1/2 \le s < 0$ we additionally need to consider $a = \beta$. 
	
	Next, we estimate the second integral. Similarly, it follows
	\begin{equation*}
		\omega_\beta(k,\xi-rk)^{2s+1}\omega_\beta(k,\xi)^{-2s} \lesssim 1+
			\abs{\xi}^{\beta}+ \abs{k}^{\frac{\beta}{\beta+1}}
	\end{equation*}
	arguing as before while using $\abs{\xi-rk} \lesssim \abs{\xi}$. The integral above with $a = 0$ and
	\begin{equation*}
		\int_0^\infty \abs{\xi}^\beta \exp(-c\abs{\xi}^\beta t) \dx t = c(\beta)
	\end{equation*}
	together with Lemma \ref{lem:upperbounde} imply $I_2(k,\xi) \le c(\beta,T)$.
	This shows the estimate on $\hat{u}_1$. 
	
	We are going to prove the estimate on $\hat{u}_2$ by introducing a suitable singular integral operator inspired by the proof of \cite[Proposition 2.3]{pruss_maximal_2004}. We denote by $\varphi$ the function $\varphi(r) = (1+r)^{1-\mu}-1$ and estimate 
	\begin{align*}
		&\int_{0}^T \int_{\R^n} \int_{\R^n} t^{2-2\mu}  \omega_\beta(k,\xi)^{2(s+1)}\abs{\hat{u}_2(t,k,\xi)}^2 \dx k \dx \xi \dx t \\
		&= \int_{0}^T \int_{\R^n} \int_{\R^n} \omega_\beta(k,\xi)^{2(s+1)} \abs{\int_0^t e_\beta(t-s,k,\xi)\varphi\left( \frac{t-s}{s}\right)s^{1-\mu}\hat{f}(s,k,\xi+(t-s)k) \dx s}^2 \dx k \dx \xi \dx t \\
		&\le \int_{0}^T \left( \int_0^t \varphi\left( \frac{t-s}{s}\right) \left( \int_{\R^n} \int_{\R^n}  \left| \omega_\beta(k,\xi-(t-s)k)^{s+1} \right. \right. \right. \\
		& \left. \left. \left. \quad e_\beta(t-s,k,\xi-(t-s)k)s^{1-\mu}\hat{f}(s,k,\xi) \right|^2 \dx k \dx \xi \right)^\frac{1}{2} \dx s \right)^2 \dx t \\
		&\lesssim \int_{0}^T \left( \int_0^t (t-s)^{-1}\varphi\left( \frac{t-s}{s}\right) \left( \int_{\R^n} \int_{\R^n} s^{2-2\mu} \omega_\beta(k,\xi)^{2s} \abs{ \hat{f}(s,k,\xi) }^2 \dx k \dx \xi \right)^\frac{1}{2} \dx s \right)^2 \dx t.
	\end{align*}
	We have first used the Minkowski integral inequality and then the estimate given by Lemma \ref{lem:decayestimsemi}. It is proven in \cite[Proposition 2.3]{pruss_maximal_2004} that the integral operator 
	\begin{equation*}
		S\colon L^2((0,T);\R) \to L^2((0,T);\R), \quad [Sw](t) = \int_{0}^t (t-s)^{-1}\varphi\left( \frac{t-s}{s}\right) w(s) \dx s
	\end{equation*}
	is bounded. We choose $w$ as the function defined by 
	$$ w(s) = \left( \int_{\R^n} \int_{\R^n} s^{2-2\mu} \omega_\beta(k,\xi)^{2s} \abs{ \hat{f}(s,k,\xi) }^2 \dx k \dx \xi \right)^\frac{1}{2}, $$
	then
	\begin{align*}
		&\int_{0}^T \int_{\R^n} \int_{\R^n} t^{2-2\mu}  \omega_\beta(k,\xi)^{2(s+1)}\abs{\hat{u}_2(t,k,\xi)}^2 \dx k \dx \xi \dx t \lesssim \int_0^T \abs{[Sw](t)}^2 \dx t  \lesssim \int_0^T \abs{w(t)}^2 \dx t \\
		&= \int_{0}^T \int_{\R^{2n}} t^{2-2\mu}  \omega_\beta(k,\xi)^{2s}\abs{\hat{f}(t,k,{\xi})}^2 \dx {\xi}   \dx k\dx t.
	\end{align*}
	Together, the estimates on $\hat{u}_1$ and $\hat{u}_2$ complete the proof of the proposition.
\end{proof} 

\begin{remark}
	Proposition \ref{prop:inhomogeneity} continues to hold for the homogeneous spaces $\dot{X}_\beta^s$ and the constant in the estimate does not depend on $T$.
\end{remark}

\begin{remark}
	We note that to get a non-integer shift in the velocity regularity one cannot simply differentiate the equation. Even for the integer regularity shift in velocity this requires a little trick. Let us sketch this. Differentiating the Kolmogorov equation once in velocity results in 
	\begin{equation*}
		\partial_t \partial_{v_i} u + v \cdot \nabla_x (\partial_{v_i}u)  = \Delta_v (\partial_{v_i} u) + \partial_{v_i} f- \partial_{x_i} u 
	\end{equation*}
	for $i = 1,\dots,n$. A priori we do not know anything about $\partial_{x_i} u$. However, differentiating the Kolmogorov equation in $x$ leads to 
	\begin{equation*}
			\partial_t D_x^\frac{1}{3} u + v \cdot \nabla_x (D_x^\frac{1}{3}u)  = \Delta_v (D_x^\frac{1}{3} u) + D_x^\frac{1}{3} f, 
	\end{equation*}
	thus assuming that $f \in L^2((0,T);X_2^{{1}/{2}})$ we deduce first $D_x u = D_x^\frac{2}{3}D_x^\frac{1}{3}u \in L^2((0,T);L^2(\R^{2n}))$ and then $u \in L^2((0,T);H^3_v(\R^{2n}))$. This shows that $u \in L^2((0,T);X_2^{{3}/{2}})$.
\end{remark}

\begin{remark}
	The reader, who is familiar with the work in \cite{bouchut_hypoelliptic_2002} could be tempted to think that it is unnecessary to prove the regularity in $x$ as it follows directly from the regularity in $v$. However, the theorem by Bouchut does only apply in the case $\mu = 1$. A version of Bouchut's theorem in $L^2$ spaces with temporal weights will be presented in Section \ref{sec:regCinf}. An extension of these results to the cases of $L^p(L^q)$ spaces with temporal weights will be considered in the forthcoming article \cite{niebel_kinetic_nodate-1}. 
\end{remark}

\begin{remark}

\label{sec:optsob}
Let us comment on the optimal gain of regularity for solutions of the Kolmogorov equation in terms of global $L^2$-estimates in the setting of strong solutions. We consider the case that $g = 0$ and let $u$ be a solution of the Kolmogorov equation. Is there a choice of parameters $\delta_1,\delta_2,\delta_3 >0$ such that
	\begin{equation*}
		\norm{D_t^{\delta_1}u}_{2,(0,\infty) \times \R^{2n}} + \norm{D_x^{\delta_2}u}_{2,(0,\infty) \times \R^{2n}} + \norm{D_v^{\delta_3}u}_{2,(0,\infty) \times \R^{2n}} \lesssim \norm{f}_{2,(0,\infty) \times \R^{2n}}?
	\end{equation*}
	Here $D_t^{\delta_1}u$ stands for an appropriately defined fractional time derivative of $u$, which enjoys the
	natural scaling law w.r.t.\ a dilation in time.
	Introducing the dilation 
\[	
\lambda \mapsto u(\lambda^\beta t,\lambda^{\beta+1}x,\lambda v)=:u_\lambda (t,x,v),
\]
the function $u_\lambda$ is a solution of the Kolmogorov equation with inhomogeneity $\lambda^\beta f_\lambda$. The scaling of the (fractional) derivatives and the definition of $u_\lambda $ together with the aforementioned $\beta$-homogeneity of the Kolmogorov equation gives
	\begin{equation*}
		\lambda^{\beta \delta_1}\norm{D_t^{\delta_1}u}_{2,(0,\infty) \times \R^{2n}}\!\!\! + \lambda^{(\beta+1)\delta_2}\norm{D_x^{\delta_2}u}_{2,(0,\infty) \times \R^{2n}}\!\!\! + \lambda^{\delta_3}\norm{D_v^{\delta_3}u}_{2,(0,\infty) \times \R^{2n}} \lesssim \lambda^\beta \norm{f}_{2,(0,\infty) \times \R^{2n}}
	\end{equation*}
	for all $\lambda >0$. This inequality can only hold if $\delta_1 = 1$, $\delta_2 =  \frac{\beta}{\beta +1 }$ and $\delta_3 = \beta$. We see that in this sense the position and velocity estimates shown in the previous section are optimal. It remains to investigate whether one can gain the full derivative in time.

\noindent One can show that the operator $\K u= -(-\Delta_v)^{\beta/2} u- v \cdot \nabla_x u$ with domain $D(\K) = \{ u \in L^2(\R^{2n}) \; | \; (-\Delta_v)^{\beta/2} u,v \cdot \nabla_x u \in L^2(\R^{2n}) \}$ is the generator of a contractive positive $C_0$-semigroup. This has been proven in \cite[Proposition 2.5]{alphonse_smoothing_2018}. Moreover, it can be shown that $i \R \subset \sigma(\K)$ using the methods introduced in \cite{metafune_lp-spectrum_2001}. This shows that the Kolmogorov semigroup cannot be analytic and thus the generator does not enjoy maximal $L^p$-regularity for any $p \in (1,\infty)$ in $L^2(\R^{2n})$. We remark that, using refined methods, it is possible to show that $\sigma(\K) = \{ z \in \C \; | \;  \mathrm{Re} (z) \le 0 \}$ as has been done in \cite{metafune_lp-spectrum_2020} in the case $\beta = 2$. 

Let us now suppose that for some $T>0$ and for every $f \in L^2((0,T) \times \R^{2n})$ we have $\partial_t u \in L^2((0,T)\times \R^{2n})$. This would imply that $v \cdot \nabla_x u \in L^2((0,T)\times \R^{2n}) $ and hence $u \in L^2((0,T), D(\K))$. Together with $u \in H^1((0,T),L^2(\R^{2n}))$ this would imply that the Kolmogorov equation satisfies the maximal $L^2$-regularity property, which is a contradiction. This shows that one can only obtain a gain of regularity in the position and in the velocity variable as given by the estimate
\begin{equation*}
  \norm{D_x^{\frac{\beta}{\beta+1}}u}_{2,(0,\infty) \times \R^{2n}} + \norm{D_v^{\beta}u}_{2,(0,\infty) \times \R^{2n}} \lesssim \norm{f}_{2,(0,\infty) \times \R^{2n}}.
\end{equation*}
\end{remark}

\section{Estimates in the homogeneous case with nonvanishing initial data}
\label{sec:homo}
In this section we are going to investigate the initial value regularity of the fractional Kolmogorov equation with $f=0$. 

\begin{lemma} \label{lem:esteta}
	Let $\beta \in (0,2]$, $\mu \in (1/2,1]$, $s \ge -1$, $T \in (0,\infty]$ and define $\eta_\beta \colon \R^{2n} \to [0,\infty)$ as
	\begin{equation*}
		\eta_\beta(k,\xi) = \int_0^T t^{2-2\mu}\omega_\beta(k,\xi-tk)^{2(s+1)}  e_\beta(t,k,{\xi}-tk)^2 \dx t. 
	\end{equation*}
	If $T \in (0,\infty)$, then the function $\eta_\beta$ satisfies the estimate
	\begin{equation*}
		 \eta_\beta(k,\xi) \le c_1 \omega_\beta(k,\xi)^{2(s+1/2-(1-\mu))}
	\end{equation*}
	for a constant $c_1 = c_1(\beta,\mu,s,T)>0$. Moreover, if $T = \infty$ the estimate
	\begin{equation*}
		 c_2 \omega_\beta(k,\xi)^{2(s+1/2-(1-\mu))} \le \eta_\beta(k,\xi)
	\end{equation*}
	holds for some constant $c_2 =c_2(\beta,\mu)>0$.
\end{lemma}

\begin{proof}
	We are going to use several integral identities, which are collected and proven at the end of the proof. Let us prove the estimate from above first. We split the integral as follows
	\begin{align*}
		\eta_\beta(k,\xi) &= \int_0^{\frac{\abs{{\xi}}}{\abs{k}} \land T} t^{2-2\mu}\omega_\beta(k,\xi-tk)^{2(s+1)}  e_\beta(t,k,{\xi}-tk)^2   \dx t \\
		&\hphantom{=} + \int_{\frac{\abs{{\xi}}}{\abs{k}}  \land T}^T t^{2-2\mu}\omega_\beta(k,\xi-tk)^{2(s+1)}  e_\beta(t,k,{\xi}-tk)^2   \dx t =: I_1+I_2.
	\end{align*}
	Using Lemma \ref{lem:upperbounde} we estimate the first integral as
	\begin{align*}
		\omega_\beta(k,\xi)^{-2(s+1/2-(1-\mu))}I_1 &= \int_0^{\frac{\abs{{\xi}}}{\abs{k}} \land T} t^{2-2\mu}\frac{(1+\abs{{\xi}-tk}^{\beta}+ \abs{k}^\frac{\beta}{\beta+1})^{2(s+1)}}{(1+\abs{{\xi}}^{\beta}+ \abs{k}^\frac{\beta}{\beta+1})^{2(s+1/2-(1-\mu))}} e_\beta(t,k,{\xi}-tk)^2   \dx t  \\
		&\lesssim \int_0^{\frac{\abs{{\xi}}}{\abs{k}} \land T} t^{2-2\mu}{(1+\abs{{\xi}}^{\beta}+ \abs{k}^\frac{\beta}{\beta+1})^{2(1/2+(1-\mu))}} e_\beta(t,k,{\xi}-tk)^2   \dx t  \\
		&\lesssim  \int_0^\infty t^{2-2\mu} \abs{{\xi}}^{2(1/2+(1-\mu))\beta}  \exp(-c \abs{{\xi}}^\beta t)   \dx t \\
		&\hphantom{=}+\int_0^\infty t^{2-2\mu} \abs{k}^{2(1/2+(1-\mu))\frac{\beta}{\beta+1}}  \exp(-c \abs{{k}}^\beta t^{\beta+1})   \dx t \\ 
		&\hphantom{=}+ \int_0^T t^{2-2\mu} \dx t =  c(\beta,\mu,T).
	\end{align*}
	The second integral can be estimated as
	\begin{align*}
		\omega_\beta(k,\xi)^{-2(s+1/2-(1-\mu))}I_2 &= \int_{\frac{\abs{{\xi}}}{\abs{k}} \land T}^T t^{2-2\mu}\frac{(1+\abs{{\xi}-tk}^{\beta}+ \abs{k}^\frac{\beta}{\beta+1})^{2(s+1)}}{(1+\abs{{\xi}}^{\beta}+ \abs{k}^\frac{\beta}{\beta+1})^{2(s+1/2-(1-\mu))}} e_\beta(t,k,{\xi}-tk)^2   \dx t  \\
		&\lesssim \int_{\frac{\abs{{\xi}}}{\abs{k}} \land T}^T t^{2-2\mu}\frac{(1+\abs{tk}^{\beta}+ \abs{k}^\frac{\beta}{\beta+1})^{2(s+1)}}{(1+\abs{{\xi}}^{\beta}+ \abs{k}^\frac{\beta}{\beta+1})^{2(s+1/2-(1-\mu))}} e_\beta(t,k,{\xi}-tk)^2   \dx t =: I_3.
	\end{align*}
	If $s+1/2-(1-\mu) \ge 0$, we estimate 
	\begin{align*}
		I_3 \lesssim \int_0^T t^{2-2\mu}\frac{(1+\abs{tk}^{\beta}+ \abs{k}^\frac{\beta}{\beta+1})^{2(s+1)}}{(1+\abs{k}^\frac{\beta}{\beta+1})^{2(s+1/2-(1-\mu))}} e_\beta(t,k,{\xi}-tk)^2   \dx t = c(\beta,\mu,s,T),
	\end{align*}
	similarly as in the proof of Proposition \ref{prop:inhomogeneity}, and if $s+1/2-(1-\mu) < 0$ we estimate 
	\begin{align*}
		I_3 \lesssim \int_0^T t^{2-2\mu}{(1+\abs{tk}^{\beta}+ \abs{k}^\frac{\beta}{\beta+1})^{2(1/2+(1-\mu))}} e_\beta(t,k,{\xi}-tk)^2   \dx t = c(\beta,\mu,T).
	\end{align*}
	
	Let us now consider the estimate from below. We treat the case $\abs{\xi}^{\beta+1} \le \abs{k}$ first. If $t \ge \frac{\abs{\xi}}{\abs{k}}$, then we have $e_\beta(t,k,\xi-tk) \ge \exp(-c \abs{k}^{\beta}t^{\beta+1})$ and hence
	\begin{align*}
		\frac{\eta_\beta(k,\xi)}{\omega_\beta(k,\xi)^{2(s+1/2-(1-\mu))}} &\ge \int_{\frac{\abs{\xi}}{\abs{k}}}^\infty t^{2-2\mu}\frac{(1+\abs{k}^{\frac{\beta}{\beta+1}})^{2(s+1)}}{( 1+\abs{\xi}^\beta+ \abs{k}^\frac{\beta}{\beta+1})^{2(s+1/2-(1-\mu))}} \exp(-c \abs{k}^{\beta}t^{\beta+1}) \dx t \\
		&\gtrsim \abs{k}^{2(1/2+(1-\mu))\frac{\beta}{\beta+1}}  \int_{\frac{\abs{\xi}}{\abs{k}}}^\infty t^{2-2\mu} \exp(-c \abs{k}^{\beta}t^{\beta+1}) \dx t \\
		&= \int_{\frac{\abs{\xi}^{\beta+1}}{\abs{k}}}^\infty t^{(3-2\mu)\frac{\beta}{\beta+1}-1} \exp(-ct) \dx t \\
		&\ge  \int_{1}^\infty t^{(3-2\mu)\frac{\beta}{\beta+1}-1} \exp(-ct) \dx t = c(\beta,\mu),
	\end{align*}
	where we have used the assumption $\abs{\xi}^{\beta+1} \le \abs{k}$ if $s+1/2-(1-\mu) \ge 0$ and omitted the $\abs{\xi}^\beta$ term in the denominator if $s+1/2-(1-\mu)<0$. Assume now that $\abs{\xi}^{\beta+1} \ge \abs{k}$, then 
	\begin{align*}
		\frac{\eta_\beta(k,\xi)}{\omega_\beta(k,\xi)^{2(s+1/2-(1-\mu))}} &\ge \int_0^{\frac{\abs{\xi}}{\abs{k}}} t^{2-2\mu} \frac{(1+\abs{\xi}^{\beta})^{2(s+1)}}{( 1+\abs{\xi}^\beta+ \abs{k}^\frac{\beta}{\beta+1})^{2(s+1/2-(1-\mu))}} \exp(-c \abs{\xi}^{\beta}t) \dx t \\
		&\gtrsim {\abs{\xi}^{\beta+2(1-\mu)\beta}} \int_0^{\frac{\abs{\xi}}{\abs{k}}} t^{2-2\mu}  \exp(-c \abs{\xi}^{\beta}t) \dx t = \int_0^{\frac{\abs{\xi}^{\beta+1}}{\abs{k}}} r^{2-2\mu} \exp(-c r) \dx r \\
		&\ge  \int_0^{1} r^{2-2\mu} \exp(-c r) \dx r = c(\mu),
	\end{align*}
	where, in the second estimate, we have used the assumption $\abs{\xi}^{\beta+1} \ge \abs{k}$ if $s+1/2-(1-\mu) \ge 0$ and omitted the $\abs{k}^\frac{\beta}{\beta+1}$ term in the denominator if $s+1/2-(1-\mu)<0$. Moreover, we have used that $e_\beta(t,k,\xi-tk) \ge \exp(-c \abs{\xi}^{\beta}t)$ for all $0 \le t \le \frac{\abs{\xi}}{\abs{k}}$. 
		
	We have used the following integral identities: 
	\begin{align*}
		&\int_0^\infty t^{2-2\mu} \abs{{\xi}}^{2(1/2+(1-\mu))\beta}  \exp(-c \abs{{\xi}}^\beta t)   \dx t = \int_0^\infty r^{2-2\mu} \exp(-cr) \dx r = c(\mu), \\
		&\int_0^\infty t^{2-2\mu} \abs{k}^{2(1/2+(1-\mu))\frac{\beta}{\beta+1}}  \exp(-c \abs{{k}}^\beta t^{\beta+1}) \dx t = \int_0^\infty r^{2-2\mu} \exp(-cr^{\beta+1}) \dx r= c(\beta,\mu), \\
		&\int_0^\infty t^{2-2\mu}\abs{tk}^{2(1/2+(1-\mu))\beta} \exp(-c \abs{{k}}^\beta t^{\beta+1})   \dx t \\
		 &= \int_0^\infty r^{2-2\mu+2(1/2+(1-\mu))\beta} \exp(-cr^{\beta+1}) \dx r = c(\beta,\mu), \\
		&\int_0^\infty t^{2-2\mu}\abs{tk}^{2(s+1)\beta}\abs{k}^{-2(s+1/2-(1-\mu))\frac{\beta}{\beta+1}} \exp(-c \abs{{k}}^\beta t^{\beta+1})   \dx t \\
		&= \int_0^\infty r^{2-2\mu+2(s+1)\beta} \exp(-cr^{\beta+1}) \dx r = c(\beta,\mu,s).
	\end{align*}	
\end{proof}

\begin{prop} \label{prop:suffcondiv}
	Let $\beta \in (0,2]$, $\mu \in (\frac{1}{2},1]$, $T \in (0,\infty)$ and $s \ge -1$. If $f = 0$ and $g \in {X}^{s+1/2-(1-\mu)}_\beta$, then $u$ given by equation \eqref{eq:fouriersol} satisfies $u \in L^2_\mu((0,T); {X}^{s+1}_\beta)$ and the estimate
	\begin{equation*}
		\int_0^T \int_{\R^{2n}} t^{2-2\mu}\omega_\beta(k,\xi)^{2(s+1)} \abs{\hat{u}(t,k,\xi)}^2 \dx \xi \dx k \dx t \le c \int_{\R^{2n}} \omega_\beta(k,\xi)^{2(s+1/2-(1-\mu))}\abs{\hat{g}(k,\xi)}^2 \dx k \dx \xi
	\end{equation*}
	is satisfied for some constant $c=c(\beta,\mu,s,T)>0$. 
\end{prop}

\begin{proof}
	We have
	\begin{align*}
		&\int_0^T \int_{\R^n} \int_{\R^n} t^{2-2\mu} \omega_\beta(k,\xi)^{2(s+1)} \abs{\hat{u}(t,k,\xi)}^2 \dx \xi \dx k \dx t \\
		&= \int_0^T \int_{\R^n} \int_{\R^n} t^{2-2\mu}\omega_\beta(k,\xi)^{2(s+1)} \abs{\hat{g}(k,\xi + tk)}^2 e_\beta(t,k,\xi)^2  \dx \xi \dx k \dx t \\
		&= \int_0^T \int_{\R^n} \int_{\R^n} t^{2-2\mu}\omega_\beta(k,\xi-tk)^{2(s+1)} \abs{\hat{g}(k,{\xi})}^2 e_\beta(t,k,{\xi}-tk)^2  \dx {\xi} \dx k \dx t \\
		&= \int_{\R^n} \int_{\R^n} \abs{\hat{g}(k,{\xi})}^2 \eta_\beta(k,{\xi}) \dx {\xi} \dx k,
	\end{align*}
	where
	\begin{equation*}
		\eta_\beta(k,\xi) = \int_0^T t^{2-2\mu}\omega_\beta(k,\xi-tk)^{2(s+1)}  e_\beta(t,k,{\xi}-tk)^2 \dx t 
	\end{equation*}
	is the same function as in Lemma \ref{lem:esteta}. We know already that $\eta_\beta(k,\xi) \lesssim \omega_\beta(k,\xi)^{2(s+1/2-(1-\mu))}$, which shows the claim. 
\end{proof}

Concerning necessary conditions on the initial value we are able to deduce the following result in the case $T = \infty$. We will later see that this result holds true in the case of finite $T$, too. 

\begin{prop} \label{prop:necessarycondiv}
	Let $\beta \in (0,2]$, $\mu \in (\frac{1}{2},1]$ and $s \ge -1$. Let $u$ be given by equation \eqref{eq:fouriersol} with $f = 0$. If $u \in L^2_\mu((0,\infty);  {X}_\beta^{s+1})$, then $g\in  {X}_\beta^{s+1/2-(1-\mu)}$,
	and the estimate
	\begin{equation*}
		 \int_{\R^{2n}} \omega_\beta(k,\xi)^{2(s+1/2-(1-\mu))}\abs{\hat{g}(k,\xi)}^2 \dx k \dx \xi \le c\int_0^\infty \int_{\R^{2n}} t^{2-2\mu}\omega_\beta(k,\xi)^{2(s+1)} \abs{\hat{u}(t,k,\xi)}^2 \dx k \dx \xi \dx t
	\end{equation*}
	holds true for some constant $c=c(\beta,\mu)>0$.
\end{prop}

\begin{proof}
	Arguing as in the proof of Proposition \ref{prop:suffcondiv} it suffices to use the estimate from below on $\eta_\beta(k,\xi)$ given in Lemma \ref{lem:esteta}. 
\end{proof}

\begin{remark}
	Analogous versions of Proposition \ref{prop:inhomogeneity}, Lemma \ref{lem:esteta}, Proposition \ref{prop:suffcondiv} and Proposition \ref{prop:necessarycondiv} can also be proven in the homogeneous setting. In particular, the estimate from above is then satisfied for all $T \in (0,\infty]$ with a constant independent of $T$.
\end{remark}

Instead of considering the spaces $ {X}_\beta^s$ one can also consider only the regularity in the position variable $x$, i.e. the space $H_x^s(\R^{2n})$.

\begin{prop} \label{prop:condivx}
	Let $\beta \in (0,2]$, $\mu \in (\frac{1}{2},1]$, $s \ge -1$ and $T \in (0,\infty)$. If $f = 0$ and 
	$g \in {H}^{(s+1/2-(1-\mu))\frac{\beta}{\beta+1}}_x(\R^{2n})$, then $u$ given by equation \eqref{eq:fouriersol} satisfies $u \in L^2_\mu((0,T);{H}^{(s+1)\frac{\beta}{\beta+1}}_x(\R^{2n}))$ and a corresponding estimate is satisfied. Conversely, if $u$ is given by equation \eqref{eq:fouriersol} with $f = 0$, then, if $u \in L^2_\mu((0,\infty); {H}^{(s+1)\frac{\beta}{\beta+1}}_x(\R^{2n}))$, it follows that $g\in {H}^{(s+1/2-(1-\mu))\frac{\beta}{\beta+1}}_x(\R^{2n})$ with the related estimate. 
\end{prop}

\begin{proof}
	Replacing $\omega_\beta(k,\xi)$ by $1+\abs{k}^\frac{\beta}{\beta+1}$ the proofs of Lemma \ref{lem:esteta}, Proposition \ref{prop:suffcondiv} and Proposition \ref{prop:necessarycondiv} continue to hold.
\end{proof}

\section{Kinetic maximal $L^2_\mu(X_\beta^s)$-regularity}
\label{sec:maxkinL2}
As sketched in Remark \ref{sec:optsob} we cannot hope that solutions of the Kolmogorov equation satisfy the maximal $L^2$-regularity property. In this section we will introduce the concept of kinetic maximal $L^2(X_\beta^s)$-regularity with temporal weights and prove that this property is satisfied for the Kolmogorov equation. 

Let $T \in (0,\infty]$, $\beta \in (0,2]$, $s \ge -\frac{1}{2}$, $\mu \in (\frac{1}{2},1]$. Of particular interest are the two cases $s = -\frac{1}{2}$ and $s = 0$, which correspond to weak and strong solutions, respectively. To describe the $L^2$ regularity of the transport term we introduce the vector space
\begin{equation*}
	\T_\mu(X_\beta^s) =\T_\mu((0,T);X_\beta^s) = \{ u \in L^2_\mu((0,T);X_\beta^s) \; \colon \partial_t u + v \cdot \nabla_x u \in L^2_\mu((0,T);X_\beta^s)  \}
\end{equation*} 
equipped with the norm $\norm{u}_{\T_\mu(X_\beta^s) } = \norm{u}_{2,\mu,X_\beta^s}+ \norm{\partial_t u + v \cdot \nabla_x u}_{2,\mu,X_\beta^s}$. To describe the spatial regularity of solutions to the Kolmogorov equation we consider the space
\begin{equation*}
	L^2_\mu(X_\beta^{s+1}) := L^2_\mu((0,T);X_\beta^{s+1}).
\end{equation*}
If $\mu = 1$ we will drop the subscript $\mu$ in the notation of the spaces.

Let us introduce the mapping $\Gamma$ acting on functions $u \colon \R\times \R^{2n} \to \R$ defined by 
$$[\Gamma u](t,x,v) = u(t,x+tv,v) $$ 
for all $t\in \R$ and any $x,v \in \R^n$. For functions $u \colon \R^{2n} \to \R$ we write $[\Gamma(t)u](x,v) = u(x+tv,v)$ for all $t \in \R$. Note that $\Gamma(t)\Gamma(-t)u=\Gamma(-t) \Gamma(t)u=u$ for all functions $u \colon \R^{2n} \to \R$.
The transformation $\Gamma$ is a key tool in our analysis of traces of functions in $\T_\mu(X_\beta^s)$. 

\begin{prop} \label{prop:GamL2isom}
	The mapping 
	\begin{equation*}
		\Gamma \colon L^2_\mu((0,T); L^2(\R^{2n})) \to L^2_\mu((0,T); L^2(\R^{2n})) 
	\end{equation*}
	is a well-defined isometric isomorphism. The inverse operator is given by $[\Gamma^{-1}u](t,x,v) = u(t,x-tv,v) $. Moreover, the group of isometries $\Gamma(t) \colon L^2(\R^{2n}) \to L^2(\R^{2n})$ is strongly continuous. The same holds true for the inverse group $(\Gamma^{-1}(t))_{t \in \R}$.
	\end{prop}

\begin{proof}
	Isometry follows by using Fubini's theorem and substituting $\tilde{x} = x +tv$. It is clear that $[\Gamma\Gamma^{-1}u](t,x,v) = u(t,x,v) = [\Gamma^{-1} \Gamma u](t,x,v)$. Finally, a proof of the strong continuity property can be found for example in \cite[Proposition 2.2]{metafune_lp-spectrum_2001}. 
\end{proof}

The proof of the latter proposition breaks down when replacing $L^2(\R^{2n})$ with $X^s_\beta$ for any $s \neq 0$ as $\Gamma$ and $D_v^s$ do not commute in case $s \in \R \setminus \{ 0 \}$. To be more precise, the calculation $\nabla_v \Gamma(t) u = \Gamma(t)\nabla_v u+t \Gamma(t) \nabla_x u$ suggests that one could consider the space $H^1(\R^{2n})$. We note that $ X^s_\beta \subset H^{-1}(\R^{2n}) $ for any $s \ge -1/2$, whence it seems natural to study the existence of a trace for functions in $\T_\mu(H^{-1}(\R^{2n}))$.
 
\begin{lemma} \label{lem:c0Z}
	Let $s \in \Z$ and $T \in (0,\infty)$. The mapping
	\begin{equation*}
		\Gamma \colon L^2_\mu((0,T);H^s(\R^{2n})) \to L^2_\mu((0,T);H^s(\R^{2n}))
	\end{equation*}
	is a well-defined unitary isomorphism. Furthermore, the group $\Gamma(t) \colon H^s(\R^{2n}) \to H^s(\R^{2n})$ is strongly continuous and the same holds true for the inverse group $(\Gamma^{-1}(t))_{t \in \R}$.
\end{lemma}

\begin{proof}
	Let us first consider the case $s = 1$. We have
	\begin{align*}
		\int_{\R^{2n}} \left( 1+ \abs{\xi}^2+\abs{k}^2 \right)^\frac{1}{2} \abs{\widehat{[\Gamma u]}(t,k,\xi)}^2 \dx k \dx \xi &= \int_{\R^{2n}} \left( 1+\abs{\xi+tk}^2+\abs{k}^2 \right)^{\frac{1}{2}} \abs{\hat{u}(t,k,\xi)}^2  \dx k \dx \xi \\
		&\le c(T) \int_{\R^{2n}} \left( 1+\abs{\xi}^2+\abs{k}^2 \right)^\frac{1}{2} \abs{\hat{u}(t,k,\xi)}^2 \dx k \dx \xi.
	\end{align*}
	Integrating in time with the temporal weight shows that $\Gamma$ is bounded. Clearly, the same holds true for the inverse operator $\Gamma^{-1}$. Moreover, $\Gamma^* = \Gamma^{-1}$ which shows that $\Gamma$ is unitary. The claim for $s = -1$ follows by the property 
	\begin{equation*}
		\langle \Gamma(t)u,w \rangle_{L^2_\mu(H^{-1})\times L^2_\mu(H^1)} = \langle u,\Gamma(-t)w \rangle_{L^2_\mu(H^{-1})\times L^2_\mu(H^1)} 
	\end{equation*}
	using the results proven already in the case $s = 1$. Here, $\langle \cdot, \cdot \rangle_{L^2_\mu(H^{-1})\times L^2_\mu(H^1)}$ denotes the dual pairing of $L^2_\mu((0,T);H^{-1}(\R^{2n})$ and $L^2_\mu((0,T);H^1(\R^{2n})$.
	
	Concerning the strong continuity of the semigroup $\Gamma$ we have in the case $s = 1$
	\begin{align*}
		&\norm{\Gamma(t)u-u}_{H^{1}(\R^{2n})} \le \norm{\Gamma(t)u-u}_{2}+ \norm{\nabla_x\Gamma(t)u - \nabla_x u}_2 + \norm{\nabla_v \Gamma(t) u- \nabla_v u}_2 \\
		&\quad\quad\le\norm{\Gamma(t)u-u}_{2}+ \norm{\Gamma(t)\nabla_x u - \nabla_x u}_2 + \norm{ \Gamma(t)\nabla_v   u- \nabla_v u}_2+t \norm{\Gamma(t)\nabla_x u}_2.
	\end{align*}
	The right-hand side of this estimate converges to zero as $t \to 0$ as a consequence of Proposition \ref{prop:GamL2isom}. The strong continuity for $s = -1$ follows by duality. The case $s = 0$ has already been proven in Proposition \ref{prop:GamL2isom}. The general case $s \in \Z$ follows similarly.  
\end{proof}

 The calculation $[\partial_t \Gamma u](t,x,v) = [\partial_t u](t,x+tv,v) + v \cdot [\nabla_x u](t,x+tv,v)  $ for sufficiently smooth functions $u$ suggests that it might be interesting to look at $\Gamma$ defined on $\T_{\mu}(H^s(\R^{2n}))$ mapping to the space
\begin{equation*}
	H^1_\mu(H^s(\R^{2n})) = \{ u \in L^2_\mu((0,T);H^s(\R^{2n})) \colon \partial_t u \in L^2_\mu((0,T);H^s(\R^{2n})) \}
\end{equation*}
equipped with the norm $\norm{u}_{H_{\mu}^1(H^s)} = \norm{u}_{2,\mu,H^s} + \norm{\partial_t u}_{2,\mu,H^s}$ for a suitable function space $H^s(\R^{2n})$.

\begin{prop} \label{prop:TtoH1}
	Let $s \in \Z$ and $T \in (0,\infty)$. The mapping
	\begin{equation*}
		\Gamma \colon \T_{\mu}((0,T);H^s(\R^{2n})) \to H^1_\mu((0,T);H^s(\R^{2n}))
	\end{equation*}
	is a well-defined isomorphism. Moreover, we have 
	$$\partial_t \Gamma u = \Gamma(\partial_t u + v \cdot \nabla_x u)$$
	in the distributional sense for all $u \in \T_{\mu}((0,T);H^s(\R^{2n}))$.
\end{prop}

\begin{proof}
	By taking the derivative $\partial_t$ of $\Gamma u$ in the distributional sense it follows that $\Gamma$ is well-defined. Furthermore, this calculation shows that $\partial_t \Gamma u = \Gamma(\partial_t u + v \cdot \nabla_x u)$ for all $u \in \T_{\mu}((0,T);H^s(\R^{2n}))$. To conclude it remains to show that $\Gamma$ is an isomorphism. As before we can explicitly give the inverse as $[\Gamma^{-1}w](t,x,v) = w(t,x-tv,v)$. Thus, it remains to show that $\Gamma^{-1} \colon H^1_\mu((0,T);H^s(\R^{2n})) \to \T_{\mu}((0,T);H^s(\R^{2n})) $ is well-defined. This follows by similar arguments and gives the identity $\partial_t \Gamma^{-1}u + v \cdot \nabla_x \Gamma^{-1}u = \Gamma^{-1} \partial_t u $.
\end{proof}

A priori it is not clear that a function $u \in \T_{\mu}((0,T);X_\beta^s) $ is continuous with values in a suitable function space and that one can define the notion of a trace $u(0)$ at $t = 0$. However, using the mapping properties of $\Gamma$ we can easily deduce the following theorem. 

\begin{theorem} \label{thm:continL2}
	Let $\beta \in (0,2]$, $s \ge -1/2$, $\mu \in (1/2,1]$ and $T \in (0,\infty)$. Then, the embedding
	\begin{equation*}
		\T_{\mu}((0,T);{X}_\beta^s) \hookrightarrow C([0,T];H^{-1}(\R^{2n})).
	\end{equation*}
	holds continuously. Moreover, if $s \ge 0$, then 
	\begin{equation*}
		\T_{\mu}((0,T);{X}_\beta^s) \hookrightarrow C([0,T];L^2(\R^{2n})).
	\end{equation*}
\end{theorem}

\begin{proof}
	Let $Z \in \{ L^2(\R^{2n}), H^{-1}(\R^{2n})\}$ and $u \in \T_{\mu}((0,T);Z) $. Then $\Gamma(u) \in H_\mu^1((0,T);Z) $ by Proposition \ref{prop:TtoH1}, and thus $ \Gamma(u) \in  C([0,T]; Z)$ as it is shown in \cite[Section 3]{pruss_maximal_2004} that 
	$$H^1_\mu((0,T); Z) \hookrightarrow C([0,T]; Z)$$ 
	continuously. In Lemma \ref{lem:c0Z} we have seen that $(\Gamma^{-1}(t))_{t \ge 0}$ is strongly continuous in $Z$ and hence also $u(t) = \Gamma^{-1}(t) ([\Gamma u](t))$ is continuous with values in $Z$ by the equicontinuity lemma. The continuity of this embedding follows by Proposition \ref{prop:TtoH1}. 
	Next, we recall that $X_\beta^s \subset H^{-1}(\R^{2n})$, which shows the first embedding. The second embedding follows from $X_\beta^s \subset L^2(\R^{2n})$ for $s \ge 0$.
\end{proof}

We deduce that for all $T \in (0,\infty] $ the vector space 
\begin{align*}
	\tr_{\mu,\mathrm{kin}}(X_\beta^s) &:= \tr (\T_\mu((0,T);X_\beta^s) \cap L^2_\mu((0,T);X_\beta^{s+1})) \\
	&:= \{ u(0) \in H^{-1}(\R^{2n}) \; | \; u \in \T_\mu((0,T);X_\beta^s) \cap L^2_\mu((0,T);X_\beta^{s+1}) \}
\end{align*}
 is well-defined by choosing $u(0)$ to be the value of the continuous representative $u$ of any function $u \in \T_\mu((0,T);X_\beta^s) \cap L^2_\mu((0,T);X_\beta^{s+1})$. We equip this trace space with the norm 
 \begin{align*}
 	\norm{g}_{\tr} &= \norm{g}_{\tr_{\mu,\mathrm{kin}}(X_\beta^s)} \\
 	&= \inf \{ \norm{u}_{\T_\mu(X_\beta^s) \cap L^2_\mu(X_\beta^{s+1})} \; | \;  u \in \T_\mu((0,T);X_\beta^s) \cap L^2_\mu((0,T);X_\beta^{s+1}) \text{ with } u(0) = g  \}.
 \end{align*}
As in the proof of \cite[Proposition 1.4.1]{amann_linear_1995} one can show that the trace space does not depend on the value of $T$. This is the reason why we dropped the $T$ dependency in our notation.

\begin{definition}
	Let $\beta \in (0,2]$, $T \in (0,\infty)$, $\mu \in (\frac{1}{2},1]$, $s \ge -\frac{1}{2}$. We say that the Kolmogorov equation of order $\beta \in (0,2]$ satisfies the kinetic maximal $L^2_\mu(X_\beta^s)$-regularity property if for all $f \in L^2_\mu((0,T);X_\beta^s)$ and any $g \in \tr_{\mu,\mathrm{kin}}(X_\beta^s)$ there exists a unique distributional solution $u \in \T_\mu((0,T);X_\beta^s) \cap L^2_\mu((0,T);X_\beta^{s+1})$ of the Kolmogorov equation 
	\begin{equation*}
			\begin{cases}
		\partial_t u +v \cdot \nabla_x u +(-\Delta_v)^{\frac{\beta}{2}} u = f, \quad t>0 \\
		u(0) = g.
	\end{cases}
	\end{equation*}
\end{definition}

\begin{remark}
	It turns out, that the conditions on $f$ and $g$ are also necessary. Indeed, if $u \in \T_\mu((0,T);X_\beta^s) \cap L^2_\mu((0,T);X_\beta^{s+1}) $ is a solution of the Kolmogorov equation, then Proposition \ref{prop:contabstract} below implies $g = u(0)  \in \tr_{\mu,\mathrm{kin}}(X_\beta^s)$. Moreover, a straightforward calculation in the Fourier variable shows $(-\Delta_v)^\frac{\beta}{2} u \in L^2_\mu((0,T);X_\beta^{s})$, whence $\partial_t u + v\cdot \nabla_x u + (-\Delta_v)^\frac{\beta}{2} u = f \in L^2_\mu((0,T);X_\beta^{s})$.
\end{remark}

\begin{remark} 
Usually, one defines maximal $L^p$-regularity without consideration of the initial value, i.e. only for the nonhomogeneous problem with initial value $u(0) = 0$. This can be done in the kinetic setting, too. A precise statement can be found in the upcoming article \cite{niebel_kinetic_nodate-1}. Furthermore, one can show that the kinetic maximal $L^p$ regularity property does not depend on $T \in (0,\infty)$. Finally, we consider more general operators instead of the fractional Laplacian in the velocity variable in \cite{niebel_kinetic_nodate-1}. 
\end{remark}

As it is the case for the classical maximal $L^2$-regularity the uniqueness property can be shown separately. 

\begin{theorem} \label{thm:uniqueweak}
	Let $T >0$, $\mu \in (\frac{1}{2},1]$, $\beta \in (0,2]$, $s \ge -1/2$. Any solution $u \in \T_\mu((0,T);X_\beta^s) \cap L^2_\mu((0,T);X_\beta^{s+1})$ of the homogeneous Kolmogorov equation $\partial_t u + v \cdot \nabla_x u = -(-\Delta_v)^\frac{\beta}{2}$, with initial value $u(0) = 0$, is equal to zero almost everywhere. 
\end{theorem}

\begin{proof}
The proof is based on the argument given in the proof of \cite[Proposition 3.2]{arendt_lp-maximal_2007}.
Let us introduce the family of operators $A_\beta(t) \colon D(A_\beta(t)) \to H^{-2}(\R^{2n})$, $\widehat{A_\beta(t)w}(k,\xi) = - \abs{\xi-tk}^\beta \hat{w}(k,\xi)$ acting on functions in $H^{-2}(\R^{2n})$ for every $t \ge 0$. We are going to show first that if $w \in H^1_\mu((0,T);H^{-2}(\R^{2n})) \cap L^2_\mu((0,T);L^2(\R^{2n}))$ is a solution of the nonautonomous problem
 	\begin{equation} \label{eq:nonautokol}
 		\begin{cases}
 			\partial_t w = A_\beta(t) w, \quad t>0 \\
 			w(0) = 0 
 		\end{cases}
 	\end{equation}
 	it follows that $w = 0$. Let us first note that 
 	\begin{align*}
 		D(A_\beta(t)) &= \{ w \in H^{-2}(\R^{2n}) \colon A_\beta(t)w \in H^{-2}(\R^{2n}) \} \\
 		&= \{ w \in H^{-2}(\R^{2n}) \colon  \int_{\R^{2n}} (1+\abs{k}^2+\abs{\xi^2})^{-2} \abs{\xi-tk}^{2\beta} \abs{\hat{w}(k,\xi)}^2 \dx k \dx \xi < \infty \},
 	\end{align*}
 	whence $L^2(\R^{2n}) \subset D(A_\beta(t))$ for all $t \in  (0,T)$. Moreover, $A_\beta(t)$ is a dissipative operator, as for all $x \in D(A_\beta(t))$ we have $ \langle A_\beta(t)x,x \rangle_{H^{-2}(\R^{2n})} \le 0$. As a consequence of $u \in H^1_\mu((0,T);H^{-2}(\R^{2n}))$ we have
 	\begin{align*}
 		\frac{\dx}{\dx t} \norm{w(t)}_{H^{-2}(\R^{2n})}^2 = 2\langle \dot{w}(t), w(t) \rangle_{H^{-2}(\R^{2n})} = 2\langle A_\beta(t) w(t), w(t) \rangle_{H^{-2}(\R^{2n})} \le 0,
 	\end{align*}
 	which shows $w = 0$.
 	 	
 	Let $u \in \T_\mu((0,T);X_\beta^s) \cap L^2_\mu((0,T);X_\beta^{s+1})$ be a solution of the Kolmogorov equation, then $u \in \T_\mu((0,T);H^{-2}(\R^{2n}) \cap L^2_\mu((0,T);L^2(\R^{2n}))$, whence $\Gamma u \in H^1_\mu((0,T);H^{-2}(\R^{2n})) \cap L^2_\mu((0,T);L^2(\R^{2n}))$ as a consequence of Lemma \ref{lem:c0Z}. A direct calculation shows that the function $\Gamma u$ satisfies the equation \eqref{eq:nonautokol}, hence it follows that $\Gamma u = 0$ and in particular this implies that $u = 0$ by Lemma \ref{lem:c0Z}.  
\end{proof}

In view of the results in Section \ref{sec:homo} it can be expected that the trace space $\tr_{\mu,\mathrm{kin}}(X_\beta^s)$ coincides with the space $X_\beta^{s+1/2-(1-\mu)}$. In what follows we will show that this is indeed the case. Let us first prove that every function $u \in \T_\mu((0,T);X_\beta^s) \cap L^2_\mu((0,T);X_\beta^{s+1})$ is continuous with values in the trace space $\tr_{\mu,\mathrm{kin}}(X_\beta^s)$.

\begin{prop} \label{prop:contabstract}
	For every $T  \in (0,\infty)$, all $s \ge -1/2$ and any $\mu \in (\frac{1}{2},1]$ we have 
	\begin{equation*}
		\T_\mu((0,T);X_\beta^s) \cap L^2_\mu((0,T);X_\beta^{s+1}) \hookrightarrow C([0,T]; \tr_{\mu,\mathrm{kin}}(X_\beta^s))
	\end{equation*}
	continuously. 
\end{prop} 

\begin{proof}
	The proof is inspired by the abstract result in \cite[Proposition 1.4.2]{amann_linear_1995}. Let us consider the case $T = \infty$ first. We introduce the translation semigroup $(\lambda_r)_{r \ge 0}$ (w.r.t.\ time) defined by $\lambda_r u = u(\cdot +r)$ on $\T_\mu((0,\infty);X_\beta^s) \cap L^2_\mu((0,\infty);X_\beta^{s+1})$. We claim that this semigroup of bounded operators is strongly continuous. This is a well-known result for the space $ L^2_\mu((0,\infty);X_\beta^{s+1})$. To see this for $\T_\mu((0,\infty);X_\beta^s) $ we note that the translation semigroup $(\lambda_r)_{r \ge 0}$ commutes with $\partial_t+v \cdot \nabla_x$. Let $r \ge 0$, then $(\lambda_r u)(0) = u(r)$. Let $u \in \T_\mu((0,\infty);X_\beta^s) \cap L^2_\mu((0,\infty);X_\beta^{s+1})$, then
	\begin{align*}
		\norm{u(t)-u(r)}_{\tr_{\mu,\mathrm{kin}}(X_\beta^s)} &= \norm{(\lambda_r(\lambda_{t-r}-1) u)(0)}_{\tr_{\mu,\mathrm{kin}}(X_\beta^s)} \le \norm{\lambda_r(\lambda_{t-r}-1) u}_{ \T_\mu(X_\beta^s) \cap L^2_\mu(X_\beta^{s+1})} \\
		& \lesssim \norm{(\lambda_{t-r}-1) u}_{ \T_\mu(X_\beta^s) \cap L^2_\mu(X_\beta^{s+1})}
	\end{align*}
	for $0\le r<t< \infty$, which shows $u \in C([0,\infty);\tr_{\mu,\mathrm{kin}}(X_\beta^s))$ as a consequence of the boundedness and strong continuity of $(\lambda_r)_{r \ge 0}$. 
	
	Let us now prove the claim of the proposition. Let $T \in (0,\infty)$ and $u \in \T_\mu((0,T);X_\beta^s) \cap L^2_\mu((0,T);X_\beta^{s+1})$. We choose $\eta \in C_c^\infty([0,\infty);\R)$ such that $\eta = 1$ in $[0,T]$ and $\eta = 0$ in $[\frac{3}{2}T,\infty)$. Then, the function 
	\begin{equation*}
			\tilde{u}(t) = 		\begin{cases}
				u(t) &\colon t \in [0,T] \\
				\eta(t)u(2T-t) &\colon t \in (T,2T] \\
				0 &\colon t \in (2T,\infty)
		\end{cases}
	\end{equation*}
	is an element of $ \T_\mu((0,\infty);X_\beta^s) \cap L^2_\mu((0,\infty);X_\beta^{s+1})$ with norm bounded by the norm of $u$ in $\T_\mu((0,T);X_\beta^s) \cap L^2_\mu((0,T);X_\beta^{s+1})$ and coincides with $u$ on $[0,T]$. From here, the claim follows. 
	
\end{proof}

We come now to the characterization of $(\tr_{\mu,\mathrm{kin}}(X_\beta^s), \norm{\cdot}_{\tr})$ in terms of anisotropic Sobolev spaces.

\begin{theorem} \label{thm:chartraceL2}
	We have
	$$(\tr_{\mu,\mathrm{kin}}(X_\beta^s), \norm{\cdot}_{\tr}) \cong X_\beta^{s+1/2-(1-\mu)},$$ 
	for all $\mu \in (\frac{1}{2},1]$, any $\beta \in (0,2]$ and all $s \ge -1/2$.
\end{theorem}

\begin{proof}
	Let $g \in X_\beta^{s+1/2-(1-\mu)}$, then according to Proposition \ref{prop:suffcondiv} the respective solution of the Kolmogorov equation with initial value $g$ and $f = 0$ satisfies $u \in \T_\mu((0,T);X_\beta^s) \cap L^2_\mu((0,T);X_\beta^{s+1})$ and the estimate 
	\begin{equation*}
		\norm{g}_{\tr} \le \norm{u}_{ \T_\mu(X_\beta^s) \cap L^2_\mu(X_\beta^{s+1})} \le C\norm{g }_{X_\beta^{s+1/2-(1-\mu)}}.
	\end{equation*} 
	We claim that the natural inclusion mapping 
	$$\iota \colon  X_\beta^{s+1/2-(1-\mu)} \to \tr_{\mu,\mathrm{kin}}(X_\beta^s)$$
	is surjective. To see this let $y \in \tr_{\mu,\mathrm{kin}}(X_\beta^s)$, i.e.\ there is a $u \in \T_\mu((0,T);X_\beta^s) \cap L^2_\mu((0,T);X_\beta^{s+1})$ such that $u (0) = y$. As a consequence of Proposition \ref{prop:necessarycondiv}, an approximation argument, the linearity of solutions and the uniqueness it must then hold $y \in X_\beta^{s+1/2-(1-\mu)}$. By the open mapping theorem the operator $\iota$ is an isomorphism. 
\end{proof}

\begin{coro} \label{cor:continHxv}
	For every $T  >0$, $\mu \in (\frac{1}{2},1]$, $\beta \in (0,2]$ and any $s \ge -1/2$ we have
	\begin{equation*}
		\T_\mu((0,T);X_\beta^s) \cap L^2_\mu((0,T);X_\beta^{s+1}) \hookrightarrow C([0,T]; X_\beta^{s+1/2-(1-\mu)}) 
	\end{equation*}
	continuously. 
\end{coro}

Let us now specify the regularization property made visible by the temporal weight $t^{1-\mu}$. For all $\delta\in (0,T)$ we have $L^2_\mu(\delta,T) \hookrightarrow L^2(\delta,T)$. This implies
\begin{align*}
	\T_\mu((0,T);X_\beta^s) \cap L^2_\mu((0,T);X_\beta^{s+1}) &\hookrightarrow \T_\mu((\delta,T));X_\beta^s) \cap L^2_\mu((\delta,T));X_\beta^{s+1}) \\
	&\hookrightarrow \T((\delta,T);X_\beta^s) \cap L^2((\delta,T);X_\beta^{s+1}) \\
	&\hookrightarrow C([\delta,T], X_\beta^{s+1/2}).
\end{align*}
We have proven the following corollary.

\begin{coro} \label{cor:gainofreg}
	For every $T  >0$, $\mu \in (\frac{1}{2},1]$, $\beta \in (0,2]$ and all $s \ge 0$ we have
	\begin{equation*}
		\T_\mu((0,T);X_\beta^s) \cap L^2_\mu((0,T);X_\beta^{s+1}) \hookrightarrow C((0,T]; X_\beta^{s+1/2}).
	\end{equation*}
\end{coro}

We are now able to state and prove the main results of this article.

\begin{theorem} \label{thm:maxreg}
	Let $T>0$ and $\mu \in (\frac{1}{2},1]$, $\beta \in (0,2]$ and $s \ge -1/2$. The fractional Kolmogorov equation of order $\beta$ satisfies the kinetic maximal $L^2_\mu(X_\beta^s)$-regularity property. 
\end{theorem}

\begin{proof}
	Let $f \in L^2_\mu((0,T);X_\beta^s)$ and $g \in \tr_{\mu,\mathrm{kin}}(X_\beta^s) \cong X_\beta^{s+1/2-(1-\mu)}$. According to Proposition \ref{prop:inhomogeneity} and Proposition \ref{prop:suffcondiv} there exists a solution $u \in \T_\mu((0,T);X_\beta^s) \cap L^2_\mu((0,T);X_\beta^{s+1})$ of the Kolmogorov equation. Uniqueness has been proven in Theorem \ref{thm:uniqueweak}. 
\end{proof}

\begin{remark}
	The reason why, in general, it only makes sense to consider the case $T< \infty$, is that the Kolmogorov semigroup is not exponentially stable and thus we cannot hope that $u \in L^2_\mu((0,\infty);L^2(\R^{2n}))$ holds in general.
\end{remark}

\section{Instantaneous Regularization of solutions to the homogeneous Kolmogorov equation}
\label{sec:regCinf}

Using the kernel representation for solutions of the Kolmogorov semigroup one can easily deduce that solutions regularize instantaneously. 
In this section, we are going to investigate this regularizing property from a more abstract point of view using in particular the kinetic maximal $L^2_\mu(X_\beta^s)$-regularity property. We also refer to \cite{lorenzi_analytical_2017}, where in case $\beta=2$ it is shown that for bounded continuous initial values the solution of the
homogeneous Kolmogorov equation is bounded and $C^\infty$-smooth for $t>0$.

\begin{prop} \label{prop:t2laplace}
	If $f = 0$ and $g \in \dot{H}_{x}^\frac{\beta/2}{\beta+1}(\R^{2n})$, then the function $u$ with Fourier transform given by equation \eqref{eq:fouriersol} satisfies the inequality
	\begin{equation*}
		\int_0^\infty \int_{\R^n} \int_{\R^n} t^{2\beta}\abs{k}^{2\beta} \abs{\hat{u}(t,k,\xi)}^2 \dx \xi \dx k \dx t \le c \int_{\R^{2n}} \abs{k}^{\frac{\beta}{\beta+1}} \abs{\hat{g}(k,\xi)}^2 \dx k \dx \xi
	\end{equation*}
	for some constant $c = c(\beta)$. In particular, we have $t^\beta (-\Delta_x)^{\beta/2} u \in L^2((0,\infty);L^2(\R^{2n}))$.
\end{prop}

\begin{proof}
	In a similar fashion as in the proof of Proposition \ref{prop:suffcondiv} we only need to estimate the function
	\begin{equation*}
		\nu(k,\xi) = \int_0^\infty t^{2\beta} \abs{k}^{2\beta} e_\beta(t,k,\xi-tk) \dx t.
	\end{equation*}
	We recall that $e_\beta(t,k,\xi-tk) \le \exp(- c\abs{k}^\beta t^{\beta+1})$ and conclude 
	\begin{equation*}
		\nu(k,\xi) \le \int_{0}^\infty t^{2\beta} \abs{k}^{2\beta} \exp(- c\abs{k}^{\beta}t^{\beta+1}) \dx t = c \abs{k}^\frac{\beta/2}{\beta+1}
	\end{equation*}
	for some constant $c>0$.
\end{proof}

The preceding proposition suggests that, when choosing the space of kinetic maximal $L^2$-regularity for strong solutions one could also take into account the integrability of $(-\Delta_x)^{\beta/2} u$ in a time weighted Lebesgue space, i.e. to describe the spatial regularity by
\begin{equation*}
	\T((0,T);L^2(\R^{2n})) \cap L^2((0,T);X_\beta^{1}) \cap L^2_{1-\beta}((0,T);\dot{H}_{x}^{\beta}(\R^{2n})).
\end{equation*}

\begin{lemma} \label{lem:continH13x}
	The family of operators $(\Gamma^{-1}(t))_{t \ge 0}$ defined by $\Gamma^{-1}(t) \colon \dot{H}_x^r(\R^{2n}) \to \dot{H}^r_x(\R^{2n})$ for any $t \in \R$ is a strongly continuous group of isometries for any $r \in [0,\infty)$. 
\end{lemma}

\begin{proof}
 This follows immediately from the fact that $\Gamma(t)$ and $D_x^r$ commute for all $r \ge 0 $.
\end{proof}

	Let $0<\delta<T$, then
\begin{align*}
	&\Gamma( \T((0,T);L^2(\R^{2n})) \cap L^2_{1-\beta}((0,T);\dot{H}_{x}^{\beta}(\R^{2n})) = H^1((0,T); L^2(\R^{2n})) \cap L^2_{1-\beta}((0,T);\dot{H}_x^\beta(\R^{2n})) \\
	&\hookrightarrow H^1((\delta,T); L^2(\R^{2n})) \cap L^2((\delta,T);H^\beta_x(\R^{2n})) \hookrightarrow C([\delta,T];H^{\beta/2}_x(\R^{2n})).
\end{align*}
	By Corollary \ref{cor:continHxv}, using Lemma \ref{lem:continH13x} and writing $u = \Gamma^{-1}\Gamma u$ it follows that 
	$$\T((0,T);L^2(\R^{2n})) \cap L^2((0,T);X_\beta^{1}) \cap L^2_{1-\beta}((0,T);\dot{H}_{x}^{\beta}(\R^{2n})) \hookrightarrow C([\delta,T];H^{\beta/2}(\R^{2n}))$$
for all $\delta >0$. We have proven the following

\begin{coro}
	For all $\beta \in (0,2]$ we have
	$$\T((0,T);L^2(\R^{2n})) \cap L^2((0,T);X_\beta^{1}) \cap L^2_{1-\beta}((0,T);\dot{H}_{x}^{\beta}(\R^{2n}))  \hookrightarrow C((0,T];H^{\beta/2}(\R^{2n})).$$
\end{coro}

In comparison to Corollary \ref{cor:continHxv} we thus gain $\frac{\beta^2}{2(\beta+1)}$ more regularity in the position variable $x$.

Unfortunately, one can not expect that solutions to the nonhomogeneous Kolmogorov equation satisfy $u \in L^2_{1-\beta}((0,T);\dot{H}_{x}^{\beta}(\R^{2n}))$ in general. A possible indication for this can be seen by considering the singular integral operator
	\begin{equation*}
		R \colon L^2((0,T);L^2(\R^{2n})) \to L^2((0,T);L^2(\R^{2n})), \; (Rf)(t) = t^\beta\int_0^t (-\Delta_x)^\frac{\beta}{2} T(s) f(t-s) \dx s, 
	\end{equation*}
	where $T(t)$ stands for the Kolmogorov semigroup.
	A direct calculation, similar to the one in Lemma \ref{lem:decayestimsemi}, shows that 
	$\norm{(-\Delta_x)^\frac{\beta}{2} T(s)}_2\lesssim s^{-(\beta+1)}$, hence the singularity of this integral close to zero cannot be controlled for larger $t$ as the factor of $t^\beta$ appears only outside of the integral. We note that we have $\norm{D_x^{\frac{\beta}{\beta+1}} T(s)}_2\lesssim s^{-1}$ which gives yet another explanation of the appearance of the $\beta/(\beta+1)$ derivatives in $x$.
	
As the gain of regularity is only in the $x$-variable, it does not suffice to show smoothness of the solutions to the homogeneous Kolmogorov equation. In order to achieve this we are going to use the regularization property of the temporal weights. 

\begin{theorem} \label{TheoremSmoothing}
	For all $T>0$, $\mu \in (\frac{1}{2},1]$, $\beta \in (0,2]$, $s \ge -1/2$ and every solution $u \in \T_\mu((0,T);X_\beta^s) \cap L^2_\mu((0,T);X_\beta^{s+1})$ of the homogeneous Kolmogorov equation we have $u \in C^\infty((0,T) \times \R^{2n})$.
\end{theorem}

\begin{proof}
	Let $\frac{1}{2}<\nu<\mu $ and $\epsilon >0$. Then $u(\epsilon) \in X_\beta^{s+1/2}$ by Corollary \ref{cor:gainofreg}, and hence $w = u(\cdot + \epsilon) \in \T_\nu((0,T-\epsilon);X_\beta^{s+(1-\nu)}) \cap L^2_\nu((0,T-\epsilon);X_\beta^{s+1+(1-\nu)})$. This shows that $u(\frac{3}{2}\varepsilon) \in X_\beta^{s+1/2+(1-\nu)}$ by Corollary \ref{cor:gainofreg}. We have proven a gain of regularity of order $(1-\nu)$. This argument can be iterated by evaluating at $\epsilon_k = \sum_{j = 0}^k 2^{-j}$ to show that $u(2\varepsilon) \in H^r(\R^{2n})$ for any $r$. We deduce that $u \in L^\infty((\delta,T];H^r(\R^{2n}))$ for any $r \ge 0$ and all $\delta>0$, hence $u \in C((0,T];H^r(\R^{2n}))$. Consequently $u \in C^{0,\infty}((0,T]\times\R^{2n})$. The equation $\partial_t = -(-\Delta_v)^{\beta/2} u -v \cdot \nabla_x u$ gives $u \in C^{1,\infty}((0,T)\times\R^{2n})$, and thus, by iteratively differentiating the Kolmogorov equation it follows that $u \in C^\infty((0,T) \times \R^{2n})$.
\end{proof}

The following result extends Bouchut's theorem (\cite[Theorem 2.1]{bouchut_hypoelliptic_2002}) to the case with temporal weights.

\begin{theorem} \label{thm:advbouchut}
	For all $r \ge 0$  and any $u \in \T_\mu((0,T);L^2(\R^{2n})) \cap L^2_\mu((0,T);H_v^r(\R^{2n}))$, we have $u \in L^2_\mu((0,T);H_x^{r/(r+1)}(\R^{2n}))$.
\end{theorem}

\begin{proof}
	For every Banach space $Z$ the mapping $\Phi_\mu \colon L^2_\mu((0,T);Z) \to L^2((0,T);Z) $, $[\Phi_\mu u] (t) = t^{1-\mu} u(t)$ defines an isomorphism. Hence, \cite[Theorem 2.1]{bouchut_hypoelliptic_2002} transfers immediately to the setting of temporal weights. 
\end{proof}

Finally, let us discuss whether the transfer of regularity from $v$ to $x$ can be seen at the level of the trace, too. Let $\beta \in (0,2]$ and $r \ge 0$. Suppose that $u \in \T_\mu((0,T);L^2(\R^{2n})) \cap L^2_\mu((0,T);H_v^r(\R^{2n}))$. This implies that $u \in L^2_\mu((0,T);H^{\frac{r}{r+1}}(\R^{2n}))$, by Theorem \ref{thm:advbouchut}, and thus $\Gamma u \in H^1_\mu((0,T);L^2(\R^{2n})) \cap L^2_\mu((0,T);H_x^{\frac{r}{r+1}}(\R^{2n}))$. By a standard interpolation argument and by \cite[Proposition 1.4.2]{amann_linear_1995} it follows that $\Gamma u \in C([0,T];H_x^{\frac{r/2-r(1-\mu)}{r+1}}(\R^{2n}))$. By Lemma \ref{lem:continH13x} and the equicontinuity lemma it follows that $u \in C([0,T];H_x^{\frac{r/2-r(1-\mu)}{r+1}}(\R^{2n}))$, whence $g = u(0) \in H_x^{\frac{r/2-r(1-\mu)}{r+1}}(\R^{2n}))$. This leads to the following observation. Regularity in $x$ of the solution is equivalent to regularity in $x$ of the initial value, whereas regularity in $v$ of the solution is equivalent to regularity in both $x$ and $v$ of the initial value. \\

\appendix

\bibliographystyle{amsplain}

\providecommand{\bysame}{\leavevmode\hbox to3em{\hrulefill}\thinspace}
\providecommand{\MR}{\relax\ifhmode\unskip\space\fi MR }
\providecommand{\MRhref}[2]{%
  \href{http://www.ams.org/mathscinet-getitem?mr=#1}{#2}
}
\providecommand{\href}[2]{#2}

\end{document}